\documentclass[12pt,letterpaper]{amsart}
\makeatletter
\@namedef{subjclassname@2020}{%
\textup{2020} Mathematics Subject Classification}
\makeatother
\usepackage{geometry,xcolor,tikz}
\usepackage{mathrsfs,latexsym,url,setspace,amssymb}
\usepackage{mathpazo} 
\usepackage[colorlinks,allcolors=blue]{hyperref} 
\usepackage{xpatch}
\xpatchcmd{\proof} 
{\itshape} 
{\bfseries} 
{}
{}
\newtheorem{thm}{Theorem}[section]
\newtheorem{prop}{Proposition}[section]
\newtheorem{cor}{Corollary}[section]
\newtheorem{lemma}{Lemma}[section]

\theoremstyle{remark}

\newcommand{\C}{\mathbb{C}}
\newcommand{\D}{\Omega}
\newcommand{\Dc}{\overline{\Omega}}
\newcommand{\ep}{\varepsilon}
\newcommand{\zb}{\overline{z}}

\onehalfspace
\allowdisplaybreaks
 \numberwithin{equation}{section}
 
\title{Compactness of Hankel and Toeplitz operators on convex 
	Reinhardt domains in $\mathbb{C}^2$}
\author{Nazl\i{} Do\u{g}an}
\address[Nazl\i{} Do\u{g}an]{Fatih Sultan Mehmet Vakif University, 
    34445 Istanbul, Turkey}
\email{ndogan@fsm.edu.tr}

\author{S\"{o}nmez \c{S}ahuto\u{g}lu}
\address[S\"{o}nmez \c{S}ahuto\u{g}lu]{University of Toledo, 
    Department of Mathematics \& Statistics, 2801 W. Bancroft, 
    Toledo, OH 43606, USA}
\email{sonmez.sahutoglu@utoledo.edu}

\date{\today}
\begin{document}

\begin{abstract} 
We study compactness of Hankel and Toeplitz operators on Bergman spaces 
of convex Reinhardt domains in $\mathbb{C}^2$ and we restrict the 
symbols to the class of functions that are continuous on the closure 
of the domain. We prove that Toeplitz operators as well as 
the Hermitian squares of Hankel operators are compact if and only 
if the Berezin transforms of  the operators vanish on the boundary 
of the domain.  
\end{abstract}

\maketitle
\section{Introduction}
Let $\D$ be a bounded domain in $\C^n$  and  $A^2(\D)$ denote the 
Bergman space on $\D$, the Hilbert space of holomorphic functions 
which are square integrable with respect to the Lebesque measure $dV$. 
Since $A^2(\D)$ is a closed subspace of $L^2(\D)$, there exists a 
bounded orthogonal projection $P:L^2(\D)\to A^2(\D)$,  called the 
Bergman projection. The Toeplitz operator $T_{\phi}:A^2(\D)\to A^2(\D)$ 
and the Hankel operator $H_{\phi}:A^2(\D)\to L^2(\D)$ are defined 
as $T_{\phi}f=P(\phi f)$ and $H_{\phi}f=(I-P)(\phi f)$,
respectively,  for $f\in A^2(\D).$ 
We define the normalized Bergman kernel as $k_z(w)=K(w,z)/\sqrt{K(z,z)}$ 
for $w,z\in \D$ where $K$ is the Bergman kernel of $\D$. Furthermore, 
the Berezin transform of a bounded linear map 
$T:A^2(\D)\to A^2(\D)$ is defined as 
$\widetilde{T}(z)=\langle Tk_z, k_z\rangle$ for $z\in \D$.

Compactness of operators in the Toeplitz algebra has been an active 
area of research. In \cite{RS24} it is shown that, for a bounded pseudoconvex 
domain $\D$ with Lipschitz boundary and $\phi\in C(\Dc)$, the Toeplitz 
operator $T_{\phi}$ is compact on  $A^2(\D)$, if and only if $\phi=0$ 
on the boundary of $\D$. This result was proven earlier on the ball  
in \cite{C73} and on the  polydisc in \cite{Le10T}. 

In this paper, we are interested in studying compactness of operators 
in terms of the boundary behavior of their Berezin transforms. 
The Axler-Zheng theorem, proven in \cite{AZ98}, states that a finite 
sum of finite products of Toeplitz operators on the Bergman space 
of the unit disc is compact if and only if the Berezin transform of 
the operator vanishes on the unit circle.  This result was generalized 
to the polydisc in \cite{E99} (and \cite[p. 232]{CKL09}), and the unit ball 
for the Toeplitz algebra in \cite[Theorem 9.5]{S07}. For further 
generalizations see \cite{MSW13,MW14,WX21}. 
The proofs of the papers mentioned in this paragraph depend on 
explicit formulas of the Bergman kernel or on its precise estimates.  

There is a second set of results that use the boundary geometry and 
several complex variables techniques, in particular the 
$\overline{\partial}$-Neumann problem, to prove Axler-Zheng type results. 
Namely, let $\mathcal{T}(\Dc)$ denote the norm closed algebra generated 
by $\{ T_{\phi} :\phi\in C(\Dc)\}$ and let $\D$ be a $C^{\infty}$-smooth 
bounded pseudoconvex domain in $\C^n$ on which the 
$\overline{\partial}$-Neumann operator is compact. Then 
$T\in \mathcal{T}(\Dc)$ is compact if and only if $\widetilde{T}=0$ 
on $b\D$ (see \cite{CS13,CS14}). The techniques used depend on the 
following geometric condition: the set of strongly pseudoconvex 
boundary points is dense in the boundary. A local and weighted 
version of the result was proven in \cite{CSZ18}. 

It is still an open question whether the Axler-Zheng  Theorem is true 
on all $C^{\infty}$-smooth bounded pseudoconvex domains in $\C^n$. 
In the hopes of discovering a new technique, in this paper, we choose 
a simple class of domains on which neither of the above mentioned 
methods can be applied. Namely, we will work on bounded convex 
Reinhardt domains  in $\C^2$.  There are no explicit formulas for  
Bergman kernels on these domains. Furthermore, the set of 
strongly pseudoconvex points is not necessarily  dense in the boundary. 
The techniques developed in this paper are inspired by 
 \cite{KZ95,Le10H,Zorboska2003}. We  obtain  an Axler-Zheng 
 type result for Toeplitz operators with symbols continuous on the 
 closure. The more general case seems more elusive at the moment.

A domain $\D\subseteq \C^n$ is called a complete Reinhardt domain 
if $\lambda\diamond  z=(\lambda_1z_1,\ldots,\lambda_nz_n)\in \D$ whenever 
$z=(z_1,\ldots,z_n)\in \D$ and 
$\lambda=(\lambda_1,\ldots,\lambda_n) \in \mathbb{D}^n$.  
One can show that convex Reinhardt domains are complete Reinhardt.

\begin{thm} \label{ThmMain1}
Let $\D$ be a bounded convex Reinhardt domain in $\C^2$  
and $\phi\in C(\overline{\D})$. Then $T_{\phi}$ is compact on 
$A^2(\D)$ if and only if $\widetilde{T}_{\phi}(z)\to 0$ as $z\to b\D$.
\end{thm}

As a corollary we get the following. 

\begin{cor} \label{CorToeplitz}
Let $\D$ be a bounded convex Reinhardt domain  in $\C^2$ 
and $\phi\in C(\overline{\D})$. Then the following are equivalent:
\begin{itemize}
\item[(i)] $T_{\phi}$ is compact.
\item[(ii)] $T^*_{\phi}T_{\phi}$ is compact.
\item[(iii)] $\widetilde{T^*_{\phi}T_{\phi}}(z)=\|T_{\phi}k_z\|^2\to 0$ 
as $z\to b\D$.
\item[(iv)] $\widetilde{T_{\phi}}(z) \to 0$ as $z\to b\D$.
\item[(v)] $\phi=0$ on $b\D$.
\end{itemize}
\end{cor}
Recently, compactness of Hankel operators has been studied 
in terms of the behavior of the symbol on the boundary of the domain. 
See, for instance, 
\cite{Le10H,ClS18,CelikSahutogluStraube20H,CelikSahutogluStraube23,Zimmer23} 
and references therein.  In this paper we study it in terms of the Berezin transform.

\begin{thm} \label{ThmMain2}
Let $\D$ be a bounded convex Reinhardt domain  in $\C^2$ 
and $\phi\in C(\overline{\D})$. Then $\|H_{\phi}k_z\|^2\to 0$ as 
$z\to b\D$ if and only if $\phi\circ f$ is holomorphic for any 
holomorphic function $f:\mathbb{D}\to b\D$. 
\end{thm}

As a corollary we have the following. 

\begin{cor}\label{CorHankel}
Let $\D$ be a bounded convex Reinhardt domain  in $\C^2$ 
and $\phi\in C(\overline{\D})$. Then the following are equivalent: 
\begin{itemize}
\item[(i)] $H_{\phi}$ is compact.
\item[(ii)] $H^*_{\phi}H_{\phi}$ is compact.
\item[(iii)] $\widetilde{H^*_{\phi}H_{\phi}}(z)=\|H_{\phi}k_z\|^2\to 0$ 
as $z\to b\D$.
\item[(iv)] $\phi\circ f$ is holomorphic for any holomorphic function 
$f:\mathbb{D}\to b\D$. 
\end{itemize}
\end{cor}

Our results show that for a subclass of positive operators, containing 
some of Hermitian squares of Hankel and Toeplitz operators, compactness 
can be characterized by the Berezin transform. However, this is not 
possible for all positive operators. For instance, the example in 
\cite{AZ98} is a positive diagonal non-compact operator whose 
Berezin transform vanishes on the boundary.

\section{Quasi-Homogeneous Functions}
 Let $\D\subseteq \C^n$ be a Reinhardt domain.  For each multi-index 
 $\alpha=(\alpha_1,\ldots,\alpha_n) \in \mathbb{Z}^n$ we define 
$e_{\alpha}(z)=z^{\alpha}/\|z^{\alpha}\|$ where  
$z^{\alpha}=z_1^{\alpha_1}\cdots z_n^{\alpha_n}$ and $\|\cdot\|$ 
denotes the $L^2$-norm on $\D$. Here we use the convention 
that $e_{\alpha}=0$ if $\|z^{\alpha}\|=\infty$.  Then there 
exists $\Gamma_{\D}\subseteq \mathbb{Z}^n$ such that 
$\{e_{\alpha}\}_{\alpha\in \Gamma_{\D}}$ is an orthonormal basis of 
$A^2(\D)$. We note that in case $\D$ is a bounded complete 
Reinhardt domain,  then $\Gamma_{\D}= \mathbb{N}_0^n$ 
where $\mathbb{N}_0=\{0,1,2,\ldots\}$. Then $K$, the Bergman 
kernel of $\D$, has the series representation 
\[ K(z,w) =\sum_{\alpha\in\Gamma_{\D}} e_{\alpha}(z)\overline{e_{\alpha}(w)}\]
converging uniformly on compact subsets of $\D\times \D$. 

We call $\phi\in L^2(\D)$ a \textit{multi-radial} function if 
$\phi(z)=\phi(|z_1|,\ldots,|z_n|)$  for $z\in \D$ and a  
\textit{quasi-homogeneous} function of multi-degree 
$J=(j_1,\ldots,j_n)\in \mathbb{Z}^n$, if $\phi(\xi\diamond z)=\xi^J\phi(z)$ 
for almost every $z\in \D$ and  $\xi\in\mathbb{T}^n$. 
If $\phi$ is quasi-homogeneous function of multi-degree $J$ 
then it  can be written as $\phi(z)=\varphi(z) e^{i J\cdot \theta}$ 
where $\varphi(z)=\phi(|z_1|,\ldots,|z_1|)$ is multi-radial function,  
$\theta_j$ is the argument of $z_j$, and 
$J\cdot \theta= j_1\theta_1+\cdots+j_n\theta_n$.

Let $\phi$ be a multi-radial function. Then $T_{\phi}$ is diagonal. 
Indeed,  we have
\begin{align*}
\langle T_{\phi}e_{\alpha},e_{\beta} \rangle 
&= \langle P(\phi e_{\alpha}),e_{\beta} \rangle 
=\langle \phi e_{\alpha},e_{\beta}\rangle  
= \int_{\D} \phi(z) e_{\alpha}(z) \overline{e_{\beta}(z)} dV(z) \\
&= \int_{\D} \phi(z) \frac{|z^{\alpha}|}{\|z^{\alpha}\|} 
e^{i\alpha \cdot \theta} \frac{|z^{\beta}|}{\|z^{\beta}\|} 
e^{-i\beta \cdot \theta} dV(z)\\
&= \begin{cases} 0 
	& \text{if}\;\; \alpha\neq \beta \\ 
	\displaystyle  \int_{\D} \phi(z) |e_{\alpha}(z)|^2 dV(z) 
	&\text{if}\;\; \alpha
	= \beta  
\end{cases}
\end{align*}
for all $\alpha, \beta\in \Gamma_{\D}$.

Next,  let $\alpha, \beta\in \Gamma_{\D}$ and 
$\phi(z)=\varphi(z)e^{iJ\cdot \theta}$ be a quasi-homogeneous 
function of multi-degree $J$ where $\varphi$ is a multi-radial 
function. Then 
\begin{align}\nonumber
\langle T_{\phi}e_{\alpha},e_{\beta} \rangle 
&= \langle P(\phi e_{\alpha}),e_{\beta} \rangle 
=\langle \phi e_{\alpha},e_{\beta}\rangle \\
\nonumber &= \int_{\D} \phi(z)e_{\alpha}(z) \overline{e_{\beta}(z)} dV(z) \\
\nonumber &= \int_{\D} \varphi(z) e^{iJ\cdot \theta} 
	e_{\alpha}(z) \overline{e_{\beta}(z)} dV(z) \\
\label{Q} &= \int_{\D} \varphi(z) e^{iJ\cdot \theta} 
\frac{|z|^{\alpha}}{\|z^{\alpha}\|} e^{i\alpha \cdot \theta}
\frac{|z|^{\beta}}{\|z^{\beta}\|} e^{-i\beta \cdot \theta} dV(z)\\
\nonumber &=  \begin{cases} 0 
	& \text{if}\;\; \alpha+J\neq \beta \\ \displaystyle \int_{\D} 
	\varphi(z) |e_{\alpha}(z)|| e_{\beta}(z)|dV(z)
	& \text{if}\;\; \alpha+J= \beta
\end{cases}.
\end{align} 
Let us denote 
$\displaystyle \lambda'_{\alpha} 
	= \int_{\D} \varphi(z) |e_{\alpha}(z)|| e_{\alpha+J}(z)|dV(z)$ 
for all $\alpha,\alpha+J\in \Gamma_{\D}$ and  
$\lambda'_{\alpha} =0$ otherwise. Then
\begin{align}\label{EqnEigT}
T_{\phi}e_{\alpha}= \lambda'_{\alpha} e_{\alpha+J}
\end{align} 
for all $\alpha\in \Gamma_{\D}$. 

To compute the eigenvalues of $H^{*}_{\phi}H_{\phi}$ we will 
use the representation
\[ H^{*}_{\phi}H_{\phi}=T_{|\phi|^{2}}-T^{*}_{\phi}T_{\phi}.\]
One can check that $T^{*}_{\phi}T_{\phi}$ is diagonal with respect to 
monomials. Hence, we have 
\begin{align*}
\langle T^{*}_{\phi} 
T_{\phi} e_{\alpha},e_{\beta} \rangle 
&= \langle T_{\phi} e_{\alpha}, T_{\phi} e_{\beta} \rangle
=\langle \lambda'_{\alpha} e_{\alpha+ J}, 
	\lambda'_{\beta} e_{\beta+J}\rangle \\
&= \begin{cases} 0 & \text{if}\;\; \alpha\neq \beta \\
	|\lambda'_{\alpha}|^2 & \text{if}\;\; \alpha= \beta 
	\text{ and } \alpha+J\in \Gamma_{\D}
\end{cases}
\end{align*}
for $\alpha, \beta \in \Gamma_{\D}$. 
We also know that $T_{|\phi|^{2}}$ is diagonal as well and 
\begin{align*}
\langle T_{|\phi|^{2}}e_{\alpha},e_{\beta} \rangle 
&= \int_{\D} |\varphi(z)|^2 \frac{|z^{\alpha}|}{\|z^{\alpha}\|_{A^{2}(\Omega)}} 
e^{i\alpha \cdot \theta} \frac{|z^{\beta}|}{\|z^{\beta}\|_{A^{2}(\Omega)}} 
e^{-i\beta \cdot \theta} dV(z)\\
&= \begin{cases} 0 & \text{ if } \alpha\neq \beta \\ 
	\displaystyle  \int_{\D} |\varphi(z)|^2 |e_{\alpha}(z)|^{2} dV(z) 
	&\text{ if }  \alpha= \beta  
\end{cases}
\end{align*}
for all $\alpha, \beta\in \Gamma_{\D}$. 

Let  
\begin{align*}
\lambda''_{\alpha}= \int_{\D} |\varphi(z)|^2 |e_{\alpha}(z)|^{2}dV(z)
\end{align*} 
for all $\alpha\in \Gamma_{\D}$. Hence we can write 
\begin{align}\nonumber 
\langle H^*_{\phi}H_{\phi}e_{\alpha},e_{\alpha}\rangle 
& =\langle T_{|\phi|^2}e_{\alpha}, e_{\alpha}\rangle 
-\langle T^*_{\phi}T_{\phi}e_{\alpha}, e_{\alpha}\rangle \\
\label{EqnEigH}& = \begin{cases} \lambda''_{\alpha} 
	& \text{ if }  \alpha+J\notin \Gamma_{\D} \\ 
	 \lambda''_{\alpha} - |\lambda'_{\alpha}|^2  
	&\text{ if } \alpha+J\in \Gamma_{\D}
\end{cases} 
\end{align}
for all $\alpha\in \Gamma_{\D}$.

Let  $H_J(\D)$ be the space of all square integrable  quasi-homogeneous 
functions of multi-degree $J$. Then, $H_J(\D)$ is a closed subspace of 
$L^2(\D)$ and $L^2(\D)=\bigoplus_{J\in\mathbb{Z}^n}H_J(\D)$. 
Let $Q_J$ be the orthogonal projection from $L^2(\D)$ to $H_J(\D)$.

The following lemma is implicitly contained in \cite{Le10H}. The proof 
can easily be adopted from the proof of \cite[Lemma 2.2]{Le10H} 
for Reinhardt domains.  
We note that $d\sigma$ below denotes the normalized surface measure 
on $\mathbb{T}^n$ so that $\sigma(\mathbb{T}^n ) = 1$. 

\begin{lemma} \label{Lem1} 
Let $\D$ be a Reinhardt domain  in $\C^n$ . Then for 
$J\in \mathbb{Z}^n,$ and $f\in L^2(\D)$, we have
\[ f_J(z) =(Q_Jf)(z)
=\int_{\mathbb{T}^n}f(z\diamond \xi) \overline{\xi^J}d\sigma(\xi)\]
for almost every $z\in \D$. Furthermore,  $H_J(\D)$ is 
orthogonal to  $H_K(\D)$ whenever $J\neq K$ and 
$L^2(\D)=\bigoplus_{J\in\mathbb{Z}^n}H_J(\D)$.
\end{lemma}

\begin{lemma}\label{L1} 
Let  $\D$ be a Reinhardt domain in $\C^n, J\in \mathbb{Z}^n, 
S\in \Gamma_{\D},$ and $\phi,\psi\in L^{\infty}(\D)$. Then
\begin{align*}
H^*_{\phi_J}H_{\psi_J}e_S(z) 
=\int_{\eta \in \mathbb{T}^n} (H^*_{\phi}H_{\psi_J}e_S)(\eta \diamond z) 
\overline{\eta}^Sd\sigma(\eta)
\end{align*}
for every $z\in \D$.
\end{lemma}

\begin{proof} 
We note that $H^*_{\phi}H_{\psi}=PM_{\overline{\phi}}H_{\psi}$ 
(see \cite[Lemma 1]{CS14}). One can verify the following equalities
\begin{align*}
H_{\psi_J}e_S(w)=&H_{\psi_J}e_S(\eta \diamond w) \overline{\eta}^{J+S}\\
K(z,w)=&K(\eta \diamond z,\eta \diamond w)
\end{align*}
for any  $J\in \mathbb{Z}^n, S\in \Gamma_{\D}, z,w\in \D$ 
and $\eta\in \mathbb{T}^n$. Then 
\begin{align*}
H^*_{\phi_J}H_{\psi_J}e_S(z)
&= PM_{\overline{\phi_J}}H_{\psi_J}e_S(z) \\
&=\int_{w\in \D} K(z,w)\overline{\phi_J(w)} (H_{\psi_J}e_S)(w)dV(w)\\
&= \int_{\eta\in \mathbb{T}^n} \int_{w\in \D} K(\eta \diamond z,\eta \diamond w)
	\overline{\phi (\eta \diamond w)} (H_{\psi_J}e_S)(\eta \diamond w) \eta^J 
	\overline{\eta}^{J+S}dV(w)d\sigma(\eta) \\
&= \int_{\eta\in \mathbb{T}^n} \int_{w\in \D} K(\eta \diamond z, w)
	\overline{\phi( w)} (H_{\psi_J}e_S)(w) \overline{\eta}^SdV(w)d\sigma(\eta)\\ 
&=\int_{\eta\in \mathbb{T}^n} PM_{\overline{\phi}}(H_{\psi_J}e_S)(\eta \diamond z) 
	\overline{\eta}^Sd\sigma(\eta) \\
&=\int_{\eta\in \mathbb{T}^n} (H^*_{\phi}H_{\psi_J}e_S)(\eta \diamond z)
	\overline{\eta}^Sd\sigma(\eta).
\end{align*}
Therefore, the proof of the lemma is complete.
\end{proof}

\begin{lemma}\label{Lem2Quasihomo}
Let  $\D$ be a Reinhardt domain in $\C^n$ such that $K(z,z)\neq 0$ 
for $z\in \D$, $J\in \mathbb{Z}^n$, and    
$\phi\in L^{\infty}(\D)$.  Then 
$\widetilde{H^*_{\phi}H_{\phi}}(z)\to 0$ as $z\to b\D$ implies 
that $\widetilde{H^*_{\phi_J}H_{\phi_J}}(z)\to 0$ as $z\to b\D$ 
where $\phi_J$ is the quasi-homogeneous part of $\phi$ 
with multi-degree $J$.
\end{lemma}

\begin{proof} 
Since
\[ \langle H^*_{\phi}H_{\phi} K(\cdot, z),e_S\rangle
=\overline{\langle H^*_{\phi}H_{\phi} e_S, K(\cdot, z)\rangle}
=\overline{ H^*_{\phi}H_{\phi} e_S(z)},\]
we can write
\[ \|H_{\phi}k_z\|^2
= \frac{\sum_{S\in \Gamma_{\D}} 
\langle H^*_{\phi}H_{\phi} K(\cdot,z), e_S \rangle e_S(z)}{K(z,z)}
=\frac{\sum_{S\in \Gamma_{\D}} 
\overline{ H^*_{\phi}H_{\phi} e_S(z)}e_S(z) }{K(z,z)}.\]
Since $H^*_{\phi_J}H_{\phi_J}$ is diagonal for every 
$J\in \mathbb{Z}^n$, there exist nonnegative $\lambda_{J,S}$ 
satisfying 
\[H^*_{\phi_J}H_{\phi_J}e_S=\lambda_{J,S}e_S\]
for every $J\in \mathbb{Z}^n$ and $S\in \Gamma_{\D}$. 
By using Lemma \ref{L1}, we get
\[\lambda_{J,S}e_S(z)=H^*_{\phi_J}H_{\phi_J}e_S(z)
=\int_{\eta\in \mathbb{T}^n} H^*_{\phi} H_{\phi_J} 
e_S(\eta \diamond z)\overline{\eta}^Sd\sigma(\eta).\]
We note that $\lambda_{J,S}=\|H_{\phi_J}e_S\|^2$. Using the 
fact that $K(z,z)=K(\eta \diamond z,\eta \diamond z)$ on Reinhardt 
domains we have
\begin{align*}
\sum_{J\in \mathbb{Z}^n} \|H_{\phi_J}k_z\|^2 
&= \sum_{J\in \mathbb{Z}^n} \sum_{S\in \Gamma_{\D}} 
	\frac{ \overline{ H^*_{\phi_J}H_{\phi_J}  e_S(z)}e_S(z) }{K(z,z)} \\
&= \sum_{J\in \mathbb{Z}^n}  \sum_{S\in \Gamma_{\D}} 
\int_{\eta\in \mathbb{T}^n} \frac{ \overline{H^*_{\phi}H_{\phi_J} 
		e_S(\eta \diamond z)}\eta^S e_S(z)}{K(\eta \diamond z,\eta \diamond z)} d\sigma(\eta) \\
&= \sum_{S\in \Gamma_{\D}}  \int_{\eta\in \mathbb{T}^n} 
	\frac{ \overline{H^*_{\phi}H_{\phi} e_S(\eta \diamond z)} 
	e_S(\eta \diamond z)}{K(\eta \diamond z, \eta \diamond z)} d\sigma(\eta) \\
&=\int_{\eta\in \mathbb{T}^n} \sum_{S\in \Gamma_{\D}}  
	\frac{ \overline{H^*_{\phi}H_{\phi} e_S(\eta \diamond z)} 
	e_S(\eta \diamond z)}{K(\eta \diamond z, \eta \diamond z)}  d\sigma(\eta)\\
&= \int_{\eta\in \mathbb{T}^n} \|H_{\phi}k_{\eta \diamond z}\|^2d\sigma(\eta).
\end{align*}
That is, we have  
\[ \|H_{\phi_J}k_z\|^2 
\leq \int_{\eta\in \mathbb{T}^n} \|H_{\phi}k_{\eta \diamond z}\|^2d\sigma(\eta) \]
for any $J\in \mathbb{Z}^n.$ 
Since the domain is Reinhardt, $\|H_{\phi}k_{\eta \diamond z}\| \to 0$ 
as $z\to b\D$ for all $\eta\in \mathbb{T}^n$. In that case,
\[ \int_{\eta\in \mathbb{T}^n} \|H_{\phi}k_{\eta \diamond z}\|^2d\sigma(\eta) 
\to 0 \text{ as }  z\to b\D.\]
That is, $\|H_{\phi_J}k_z\|=\widetilde{H^*_{\phi_J}H_{\phi_J}}(z)\to 0$ 
as $z\to b\D$ for any $J\in \mathbb{Z}^n$. 
\end{proof}

\begin{lemma}\label{LemFourierBerezin}
Let  $\D$ be a Reinhardt domain in $\C^n$ such that 
$K(z,z)\neq 0$ for $z\in \D$, $J\in \mathbb{Z}^n,$ and $\phi\in L^{\infty}(\D)$. 
Then $\widetilde{(\phi_J)}=(\widetilde{\phi})_J$.	
\end{lemma} 
\begin{proof}
Let $z\in \D$. Then 
\begin{align*}
\widetilde{(\phi_J)}(z)
=&\int_{\D}\phi_J(w)\frac{|K(w,z)|^2}{K(z,z)}dV(w)\\
=&\int_{w\in \D}\int_{\eta\in \mathbb{T}^n} \phi(\eta\diamond w)
	\overline{\eta}^J\frac{|K(w,z)|^2}{K(z,z)}d\sigma(\eta)dV(w)\\
=&\int_{w\in \D}\int_{\eta\in \mathbb{T}^n} \phi(\eta\diamond w)
	\overline{\eta}^J\frac{|K(\eta\diamond w,\eta\diamond z)|^2}{K(\eta\diamond z,
	\eta\diamond z)}d\sigma(\eta)dV(w)\\
=&\int_{\eta\in \mathbb{T}^n} \int_{w\in \D} \phi(\eta\diamond w)
\overline{\eta}^J\frac{|K(\eta\diamond w,\eta\diamond z)|^2}{K(\eta\diamond z,
	\eta\diamond z)}dV(w)d\sigma(\eta)\\
=&\int_{\eta\in \mathbb{T}^n} \int_{\xi\in \D} \phi(\xi)
	\overline{\eta}^J\frac{|K(\xi,\eta\diamond z)|^2}{K(\eta\diamond z,
	\eta\diamond z)}dV(\xi)d\sigma(\eta)\\
=&\int_{\eta\in \mathbb{T}^n} \widetilde{\phi}(\eta\diamond z)
	\overline{\eta}^Jd\sigma(\eta)\\
=&(\widetilde{\phi})_J(z).
\end{align*} 
Therefore,  $\widetilde{(\phi_J)}=(\widetilde{\phi})_J$.	
\end{proof} 

Let  $\D$ be a bounded Reinhardt domain in $\C^n$ and  
$\phi\in L^2(\D)$. We define  the Ces\`aro mean $\Lambda_k(\phi )$  as 
\[\Lambda_k(\phi)=\sum_{|j_1 |\leq k,\ldots,|j_n |\leq k}
\left(1-\frac{|j_1|}{k+1}\right)\cdots\left(1-\frac{|j_n|}{k+1}\right) \phi_J\]
for $k\in \mathbb{N}$ (see \cite[Sec. 2.5]{KBook}). 

Next we state a well known result in harmonic analysis and we provide 
a proof here for the convenience of the reader. 

\begin{lemma}\label{LemSum}
Let  $\D$ be a bounded Reinhardt domain in $\C^n$ and  $\phi\in C(\Dc)$. 
Then $\Lambda_k(\phi)\to\phi$ uniformly on $\Dc$ as $k\to\infty$.
\end{lemma}
\begin{proof}
Let us denote  $\mathbb{T}^n_{\delta} 
= \{\eta\in \mathbb{T}^n:|\eta_j|<\delta \text{ for } 1\leq j\leq n\}$ 
for $\delta>0$ and $F_k$ denote the $k$th Fej\'er’s kernel.  
In the proof we will use some facts from sections 2.2 and 2.5 in 
\cite{KBook}. First, $F_k$ is a positive summability kernel. 
Then for any $\delta>0$ we have 
\begin{align}\label{EqFej}
\lim_{k\to\infty}\int _{\mathbb{T}^n\setminus \mathbb{T}^n_{\delta}} 
F_k(\eta_1)\cdots F_k(\eta_n)d\sigma(\eta)=0
\end{align} 
Furthermore,  
\begin{align*}
\Lambda_k(\phi)(z)
=\int_{\mathbb{T}^n} F_k(\eta_1)\cdots F_k(\eta_n)
\phi(\eta \diamond z)d\sigma(\eta)
\end{align*} 
for $k \geq 1$. 
Since $\phi$ is uniformly continuous on $\Dc$, for $\ep>0$ 
there exists $\delta>0$ such that $|\phi(\eta \diamond z)-\phi(z)|<\ep/2$ 
for $\eta\in \mathbb{T}^n_{\delta}$ and $z\in \Dc$.  Then 
\begin{align*}
|\Lambda_k(\phi)(z)-\phi(z)|
=&\, \int _{\mathbb{T}^n} F_k(\eta_1)\cdots F_k(\eta_n)
	(\phi(\eta \diamond z)-\phi(z))d\sigma(\eta) \\
\leq &\,  \int _{\mathbb{T}^n_{\delta}} F_k(\eta_1)\cdots F_k(\eta_n)
	|\phi(\eta \diamond z)-\phi(z)|d\sigma(\eta) \\
&+  \int _{\mathbb{T}^n\setminus \mathbb{T}^n_{\delta}} 
	F_k(\eta_1)\cdots F_k(\eta_n)
	|\phi(\eta \diamond z)-\phi(z)|d\sigma(\eta)\\
\leq &\, \frac{\ep}{2}+2\|\phi\|_{L^{\infty}}
	\int _{\mathbb{T}^n\setminus \mathbb{T}^n_{\delta}} 
	F_k(\eta_1)\cdots F_k(\eta_n)d\sigma(\eta).
\end{align*} 	
Then by \eqref{EqFej} there exists $N\in\mathbb{N}$ such that 
\[ |\Lambda_k(\phi)(z)-\phi(z)|\leq \ep\]
for $z\in \Dc$ and $k\geq N$. Therefore, $\Lambda_k(\phi)\to \phi$ 
uniformly on $\Dc$ as $k\to\infty$.
\end{proof} 

\begin{cor}\label{CorZero} 
Let  $\D$ be a bounded Reinhardt domain in $\C^n$ and 
$\phi\in L^{\infty}(\D)$ such that $\widetilde{\phi}\in C(\Dc)$. 
Then $\widetilde{\phi}=0$ on $b\D$ if and only if  
$\widetilde{(\phi_J)}=0$ on $b\D$ for all $J\in \mathbb{Z}^n$.	 
\end{cor}
\begin{proof}
First we assume that $\widetilde{\phi}=0$ on $b\D$. Then 
Lemma \ref{LemFourierBerezin} implies that 
$\widetilde{(\phi_J)}=(\widetilde{\phi})_J=0$ on $b\D$ 
for all $J\in \mathbb{Z}^n$.	 

For the converse direction, let $\widetilde{(\phi_J)}=0$ on $b\D$ for all 
$J\in \mathbb{Z}^n$. Then again by Lemma \ref{LemFourierBerezin} 
we have  $(\widetilde{\phi})_J=\widetilde{(\phi_J)}=0$ on $b\D$. 
Hence  $\Lambda_k(\widetilde{\phi})=0$ on $b\D$ for all $k$. 
Lemma \ref{LemSum} implies that  $\Lambda_k(\widetilde{\phi})\to \phi$ 
uniformly on $\Dc$. Therefore, $\widetilde{\phi}=0$ on $b\D$.
\end{proof} 		

\section{Proof of Theorem \ref{ThmMain1} 
	and Corollary \ref{CorToeplitz} }
We will need the following result (see \cite[Theorem 2, p. 32]{Postkinov}).
\begin{thm} \label{T1}
Let $\sum_{k=0}^{\infty}b_kt^k$ be convergent for $|t|<1$. 
Assume that there exist $\lambda\geq 0$ and $C>0$ such that 
$b_k\geq -C(k+1)^{(\lambda-1)}$ for all $k$ and 
\[\lim_{t\to1^-}(1-t)^{\lambda}\sum_{k=0}^{\infty}b_kt^k=0.\]
Then 
\[\lim_{n\to\infty}\frac{\sum_{k=0}^nb_k}{(n+1)^{\lambda}}=0.\]
\end{thm}
We say a set $K\subseteq \C^2$ contains a vertical non-trivial 
analytic disc if there exist $a\in \C$ and $r>0$ such that 
$\{(a,\xi):|\xi|<r\}\subset K$. Similarly, $K$ contains a horizontal 
non-trivial analytic disc if $\{(\xi,a):|\xi|<r\}\subset K$. 

\begin{lemma}\label{LemToepQhomo}
Let $\D$ be a bounded convex Reinhardt domain  in $\C^2$ 
and  $\phi\in C(\Dc)$ be a quasi-homogeneous function with 
multi-degree $J\in \mathbb{Z}^2$. Assume that $b\D$ 
contains a non-trivial vertical analytic disc and 
$\widetilde{T_{\phi}}=0$ on $b\D$. Then $\lambda'_{\alpha}\to 0$ 
as  $\alpha_1\to\infty$ for all  $\alpha_2\in \mathbb{N}_0$ 
where $T_{\phi}e_{\alpha}=\lambda'_{\alpha}e_{\alpha+J}$ for  
$\alpha \in \mathbb{N}_0^2$. 
\end{lemma} 
	
\begin{proof}
By assumption, $b\D$ contains a non-trivial vertical analytic disc.  
In the interest of simplifying the computations below, 
without loss of generality (we apply a dilation, if necessary), 
we assume that 
\[\mathbb{D}^2=\{z\in \C^2: |z_1|<1, |z_2|< 1\}  
\text{ and } U=\{z\in \C^2: |z_1|< 1, |z_2|< s\}\]
such that $\mathbb{D}^2\subset \D\subset U$ (see Figure \ref{Fig}). 
\begin{figure}[t] 
\scalebox{0.8}{
\begin{tikzpicture}
\draw[thick,->] (0,0) -- (6,0) node[anchor=north west] {};
\draw[thick,->] (0,0) -- (0,6) node[anchor=south east] {};
\node[draw=white] at (-0.5,6) {$|z_2|$};
\draw (4,0) -- (4,2);
\draw (4,2) arc (45:88:6cm);
\draw [fill] (4,2) circle[radius=2pt];
\draw [fill] (0,3.75) circle[radius=2pt];
\draw [fill] (4,0) circle[radius=2pt];
\node[draw=white] at (-0.45,3.75) {$s$};
\node[draw=white] at (5,2) {$(1,1)$};
\node[draw=white] at (4,-0.4) {$1$};
\node[draw=white] at (6,-0.4) {$|z_1|$};
\node[draw=white] at (1.9,2.5) {$\Omega$};
\end{tikzpicture}}
\caption{Convex Reinhardt Domain} \label{Fig}
\end{figure}
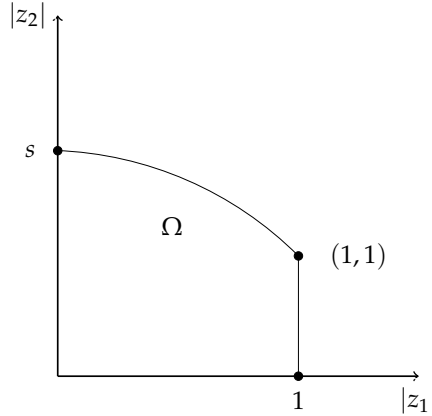 
We note that $\lambda'_{\alpha}=0$ if $\alpha+J\not\in \mathbb{N}_0^2$ 
and by \eqref{EqnEigT}  
\[\lambda'_{\alpha}= \int_{\D} \varphi(z) |e_{\alpha}(z)|| e_{\alpha+J}(z)|dV(z)\] 
$\alpha+J\in \mathbb{N}_0^2$.
We calculate the Berezin transform of $T_{\phi}$ at $z\in\mathbb{D}^2$.
\begin{align*}
\widetilde{T_{\phi}}(z)
&=\langle T_{\phi}k_z,k_z\rangle 
= \frac{1}{K(z,z)} \langle T_{\phi} K(\cdot,z), K(\cdot,z) \rangle \\
& =\frac{1}{K(z,z)} \sum_{\alpha,\beta \in \mathbb{N}^2_0} 
	\overline{e_{\alpha}(z)} e_{\beta}(z) \langle T_{\phi}e_{\alpha}, 
	e_{\beta}\rangle\\
&= \frac{1}{K(z,z)}\sum_{\alpha\in \mathbb{N}^2_0}
\frac{\zb^{\alpha}z^{\alpha+J}}{\|z^{\alpha}\|\|z^{\alpha+J}\|}\lambda_{\alpha}\\
& =\frac{K_{\mathbb{D}^2}(z,z)}{K(z,z)} \frac{z^J}{K_{\mathbb{D}^2}(z,z)} 
	\sum_{\alpha\in \mathbb{N}^2_0}
	\frac{|z^{\alpha}|^2}{\|z^{\alpha}\|\|z^{\alpha+J}\|}\lambda_{\alpha}
\end{align*}
where $K_{\mathbb{D}^2}$ denotes the Bergman kernel of $\mathbb{D}^2$. 
We note that \cite[Theorem 12.1.29]{JPBook} implies that for $|z_2|<1$ fixed 
there exists $C>1$ such that 
$1\leq K_{\mathbb{D}^2}(z,z)/K(z,z)\leq C$ for $|z_1|<1$.  
Then $\widetilde{T_{\phi}}(z)\to 0$ as $z\to b\D$ implies that for any 
fixed $z_2$ with $|z_2|<1$ we have 
\[ \frac{z^J}{K_{\mathbb{D}^2}(z,z)}  \sum_{\alpha\in \mathbb{N}^2_0}
\frac{|z^{\alpha}|^2}{\|z^{\alpha}\|\|z^{\alpha+J}\|}
\lambda'_{\alpha}  \to 0  \text{ as } |z_1|\to 1^-.\]

We write 
\begin{align*}
&\frac{z^J}{K_{\mathbb{D}^2}(z,z)}  \sum_{\alpha\in \mathbb{N}^2_0}
	\frac{|z^{\alpha}|^2}{\|z^{\alpha}\|\|z^{\alpha+J}\|} 
	\lambda'_{\alpha} \\  
&=\pi^2z^J(1-|z_1|^2)^2(1-|z_2|^2)^2 	
	\sum_{\alpha\in \mathbb{N}^2_0} 
	\frac{|z_1|^{2\alpha_1}|z_2|^{2\alpha_2}}{\|z^{\alpha}\|\|z^{\alpha+J}\|}
	\lambda'_{\alpha} \\ 
&=z^J \left(1-|z_1|^2\right)^2 
	\sum_{\alpha_1\in \mathbb{N}_0}(\alpha_1+1)A(\alpha_1, |z_2|) 
	|z_1|^{2\alpha_1} 
\end{align*}
where
\[ A(\alpha_1,|z_2|) =\pi^2 (1-|z_2|^2)^2 
\sum_{\alpha_2\in \mathbb{N}_0} \frac{|z_2|^{2\alpha_2}}{(\alpha_1+1) 
	\|z^{\alpha}\|\|z^{\alpha+J}\|}\lambda'_{\alpha}.\]
Let us denote $t=|z_1|^2$. Then 
\[\lim_{t\to 1^{-}}(1-t)^2
\sum_{\alpha_1\in  \mathbb{N}_0} (\alpha_1+1)A(\alpha_1,|z_2|)t^{\alpha_1}=0.\]

We will use Theorem \ref{T1} to conclude that $A(\alpha_1,|z_2|) \to 0$ 
	as $\alpha_1\to\infty$. To that end, we rewrite the sum above as 
\begin{align*}
(1-t)^2
&\sum_{\alpha_1\in  \mathbb{N}_0} (\alpha_1+1)A(\alpha_1,|z_2|)t^{\alpha_1} 
	=(1-t)A(0,|z_2|)\\
&+ (1-t)\left(\sum_{\alpha_1\in  \mathbb{N}} 
 	\left((\alpha_1+1)A(\alpha_1,|z_2|)- 
	\alpha_1A(\alpha_1-1,|z_2|)\right)t^{\alpha_1} \right) 
\end{align*} 
and we want to show that 
\begin{align}\label{EqnBdd1} 
\sup\{|(\alpha_1+1)A(\alpha_1,|z_2|)- \alpha_1A(\alpha_1-1,|z_2|)|:
\alpha_1\in  \mathbb{N}\}<\infty
\end{align}  
for any fixed $z_2\in \mathbb{D}$. We note that 
\begin{align*}
&|(\alpha_1+1)A(\alpha_1,|z_2|)- \alpha_1A(\alpha_1-1,|z_2|)| \\
&\leq
	\pi^2 (1-|z_2|^2)^2 
	\sum_{\alpha_2\in \mathbb{N}_0} |z_2|^{2\alpha_2} 
	\left|\frac{\lambda'_{\alpha_1,\alpha_2}}{\|z^{\alpha}\|\|z^{\alpha+J}\|}
	-\frac{\lambda'_{\alpha_1-1,\alpha_2}}{\|z^{(\alpha_1-1,\alpha_2)}\|
	\|z^{(\alpha_1-1,\alpha_2)+J}\|}\right|. 
\end{align*} 
Then to prove \eqref{EqnBdd1} it is enough to show that 
\begin{align}\label{EqnLambda'}
\sup\left\{\left|\frac{\lambda'_{\alpha_1,\alpha_2}}{
	\|z^{\alpha}\|\|z^{\alpha+J}\|}
-\frac{\lambda'_{\alpha_1-1,\alpha_2}}{\|z^{(\alpha_1-1,\alpha_2)}\|
	\|z^{(\alpha_1-1,\alpha_2)+J}\|}\right|:
\alpha_1\in  \mathbb{N}\right\}<\infty.
\end{align}  

Below we use the notation $x=|z_1|,$ and $ y=|z_2|$. 
We also choose $\rho_1$ so that  the inequality
$y<\rho_1(x)$ defines $|\D|=\{(|z_1|,|z_2|)\in\mathbb{R}^2:z\in \D\}$, 
the shadow of the domain $\D$ 
(the representation of the domain in the absolute space). 
Since $\lambda'_{\alpha}=0$ for $\alpha+J\not\in\mathbb{N}_0^2$, 
in the following computations we will assume that 
$\alpha+J\in\mathbb{N}_0^2$.
\begin{align*}
\lambda'_{\alpha} 
=&\, \frac{1}{\|z^{\alpha}\|
	\|z^{\alpha+J}\|}\int\varphi(z)|z^{\alpha}||z^{\alpha+J}|dV(z) \\
=&\,\frac{4\pi^2}{\|z^{\alpha}\|\|z^{\alpha+J}\|}
	\int_0^1x^{2\alpha_1+j_1+1} 
	\int_0^{\rho_1(x)}\varphi(x,y)y^{2\alpha_2+j_2+1}dydx\\
=&\,\frac{4\pi^2}{\|z^{\alpha}\|\|z^{\alpha+J}\|}
	\int_0^1x^{2\alpha_1+j_1+1} \psi'(x,\alpha_2)dx
\end{align*}
where 
\begin{align}\label{EqnPsi'}
\psi'(x,\alpha_2)= \int_0^{\rho_1(x)}\varphi(x,y)y^{2\alpha_2+j_2+1}dy.
\end{align}  
We note that $\psi'(\cdot,\alpha_2)$ is continuous on $[0,1]$. 

Next we want to show that 
\begin{align}\label{EqnRatio}
\lim_{\alpha_1\to\infty}\frac{(\alpha_1+1)(\alpha_2+1)}{\pi^2} 
\|z^{\alpha}\|^2  =1.
\end{align} 
Let $0<\ep<1$. Since $\rho_1$ is a positive decreasing function we have 
\begin{align*}
\|z^{\alpha}\|^2=
&\, 4\pi^2\int_0^1\int_0^{\rho_1(x)}x^{2\alpha_1+1}y^{2\alpha_2+1}dydx\\
= &\, \frac{2\pi^2}{\alpha_2+1}\int_0^1x^{2\alpha_1+1}
	(\rho_1(x))^{2\alpha_2+2}dx\\
\leq &\frac{2\pi^2}{\alpha_2+1}\int_0^{1-\ep}x^{2\alpha_1+1}
	(\rho_1(x))^{2\alpha_2+2}dx
	+ \frac{2\pi^2}{\alpha_2+1}\int_{1-\ep}^1x^{2\alpha_1+1}
	(\rho_1(1-\ep))^{2\alpha_2+2}dx\\
\leq & \frac{2\pi^2}{\alpha_2+1}\int_0^{1-\ep}x^{2\alpha_1+1}
	(\rho_1(x))^{2\alpha_2+2}dx
	+\frac{\pi^2}{(\alpha_1+1)(\alpha_2+1)}
	(\rho_1(1-\ep))^{2\alpha_2+2}.
\end{align*}
Then 
\[\limsup_{\alpha_1\to\infty}\frac{(\alpha_1+1)(\alpha_2+1)}{\pi^2} 
\|z^{\alpha}\|^2  \leq (\rho_1(1-\ep))^{2\alpha_2+2}.\]
Then, the fact that 
$\|z^{\alpha}\|^2\geq \|z^{\alpha}\|^2_{L^2(\mathbb{D}^2)}
=\pi^2/(\alpha_1+1)(\alpha_2+1)$ and $\rho_1(1)=1$ 
imply \eqref{EqnRatio}.

Next we estimate 
\begin{align*}
&  \frac{\lambda'_{\alpha_1,\alpha_2}}{
	\|z^{\alpha}\|\|z^{\alpha+J}\|}
	-\frac{\lambda'_{\alpha_1-1,\alpha_2}}{\|z^{(\alpha_1-1,\alpha_2)}\|
	\|z^{(\alpha_1-1,\alpha_2)+J}\|}\\  
=&\frac{4\pi^2}{\|z^{\alpha}\|^2\|z^{\alpha+J}\|^2}
	\int_0^1x^{2\alpha_1+j_1+1} \psi'(x,\alpha_2)dx \\
&-\frac{4\pi^2}{\|z^{(\alpha_1-1,\alpha_2)}\|^2
	\|z^{(\alpha_1-1,\alpha_2)+J}\|^2}
	\int_0^1x^{2\alpha_1+j_1-1} \psi'(x,\alpha_2)dx\\
=&\frac{4\pi^2}{\|z^{\alpha}\|^2\|z^{\alpha+J}\|^2}
	\int_0^1(x^{2\alpha_1+j_1+1}-x^{2\alpha_1+j_1-1}) 
	\psi'(x,\alpha_2)dx\\
&+\left(\frac{4\pi^2}{\|z^{\alpha}\|^2\|z^{\alpha+J}\|^2}
	-\frac{4\pi^2}{\|z^{(\alpha_1-1,\alpha_2)}\|^2
	\|z^{(\alpha_1-1,\alpha_2)+J}\|^2}\right)
	\int_0^1x^{2\alpha_1+j_1-1} \psi'(x,\alpha_2)dx.
\end{align*}
The first integral on the right hand side of the last equality above 
is bounded because $\psi'$ is bounded, 
$x^{2\alpha_1+j_1+1}\leq x^{2\alpha_1+j_1-1}$ on $[0,1]$, 
\[\int_0^1(x^{2\alpha_1+j_1+1}-x^{2\alpha_1+j_1-1})dx 
= -\frac{2}{(2\alpha_1+j_1+2)(2\alpha_1+j_1)}\] 
and, by \eqref{EqnRatio}, 
$(\alpha_1+1)\|z^{\alpha}\|^2\to\pi^2/(\alpha_2+1)$ as $\alpha_1\to\infty$. 

We will use the notation $A(\alpha)\approx \alpha_1$ below 
to mean that there exists $C,D>0$ independent of $\alpha_1$ 
such that for fixed $\alpha_2$ we have 
$C\alpha_1\leq A(\alpha)\leq D\alpha_1$ for all $\alpha_1$.
Then to verify \eqref{EqnBdd1}, it is enough to show that 
\[\frac{4\pi^2}{\|z^{\alpha}\|^2\|z^{\alpha+J}\|^2}
-\frac{4\pi^2}{\|z^{(\alpha_1-1,\alpha_2)}\|^2
\|z^{(\alpha_1-1,\alpha_2)+J}\|^2} \approx \alpha_1.\] 
This is the same as showing that 
\begin{align*} 
 \|z^{(\alpha_1-1,\alpha_2)}\|^2\|z^{(\alpha_1-1,\alpha_2)+J}\|^2
	-  \|z^{\alpha}\|^2\|z^{\alpha+J}\|^2
\approx \alpha_1^{-3}
\end{align*}
as \eqref{EqnRatio} implies that 
$ \|z^{(\alpha_1-1,\alpha_2)}\|^2\|z^{(\alpha_1-1,\alpha_2)+J}\|^2
\|z^{\alpha}\|^2\|z^{\alpha+J}\|^2\approx \alpha_1^{-4}$. 
To that end, for a fixed $\alpha_2$, we compute 
\begin{align}\nonumber 
&\|z^{(\alpha_1-1,\alpha_2)}\|^2\|z^{(\alpha_1-1,\alpha_2)+J}\|^2
	- \|z^{\alpha}\|^2\|z^{\alpha+J}\|^2\\
\nonumber 	\approx &\left(\int_0^1x^{2\alpha_1-1}
	(\rho_1(x))^{2\alpha_2+2}dx\right) \left(\int_0^1x^{2\alpha_1+2j_1-1}
	(\rho_1(x))^{2\alpha_2+2j_2+2}dx\right) \\
\nonumber	& -\left(\int_0^1x^{2\alpha_1+1}
	(\rho_1(x))^{2\alpha_2+2}dx\right)\left(\int_0^1x^{2\alpha_1+2j_1+1}
	(\rho_1(x))^{2\alpha_2+2j_2+2}dx\right) \\
\nonumber	=&\left(\int_0^1(x^{2\alpha_1-1}-x^{2\alpha_1+1})
	(\rho_1(x))^{2\alpha_2+2}dx\right)\left(\int_0^1x^{2\alpha_1+2j_1-1}
	(\rho_1(x))^{2\alpha_2+2j_2+2}dx\right)\\
\nonumber	& +\left(\int_0^1x^{2\alpha_1+1}
	(\rho_1(x))^{2\alpha_2+2}dx\right)
	\left(\int_0^1(x^{2\alpha_1+2j_1-1}-x^{2\alpha_1+2j_{1}+1})
	(\rho_1(x))^{2\alpha_2+2j_2+2}dx \right)\\
\label{EqnNormDiff} 	\approx & \, \alpha_1^{-2}\alpha_1^{-1}
	+\alpha_1^{-1}\alpha_1^{-2}\approx\alpha_1^{-3}.
\end{align}
 Hence, we verified  \eqref{EqnBdd1}. 
 
 Then  we  use Theorem \ref{T1} on 
 $(1-t)\sum_{k=0}^{\infty}b_kt^k$  where $b_0=A(0,|z_2|)$ and 
 $b_k=(k+1)A(k,|z_2|)- kA(k-1,|z_2|)$ for $k\geq 1$ to conclude  
 \[\frac{\sum_{k=0}^{\alpha_1}b_k}{\alpha_1+1}
 =A({\alpha_1},|z_2|)\to 0 \text{ as } \alpha_1\to \infty\]
 for any fixed $z_2\in \mathbb{D}$. 

Next,  we will show that $\{A(\alpha_1, \cdot)\}$ 
is uniformly bounded depending on $j_2$. Below we use the facts that 
$\|z^{\alpha+J}\|_{A^2(\mathbb{D}^2)}\leq \|z^{\alpha}\|$ and 
$|\lambda'_{\alpha}|\leq \|T_{\phi}\| $. 
\begin{align}\nonumber 
|A(\alpha_1,|z_2|)|  
&\leq \pi^2 (1-|z_2|^2)^2 
\sum_{\alpha_2\in \mathbb{N}_0} \frac{|z_2|^{2\alpha_2}}{(\alpha_1+1) 
	\|z^{\alpha}\|_{A^2(\mathbb{D}^2)}\|z^{\alpha+J}\|_{A^2(\mathbb{D}^2)}}
|\lambda'_{\alpha}| \\
\nonumber &\leq \frac{\|T_{\phi}\| (1-|z_2|^2)^2}{\sqrt{\alpha_1+1}}	
\sum_{\alpha_2\in \mathbb{N}_0} 
\sqrt{(\alpha_2+1)(\alpha_1+j_1+1)(\alpha_2+j_2+1)}|z_2|^{2\alpha_2}\\
\label{EqnBdd}&\leq \frac{\|T_{\phi}\| 
	\sqrt{\alpha_1+j_1+1}}{\sqrt{\alpha_1+1}}(1-|z_2|^2)^2	
\sum_{\alpha_2\in \mathbb{N}_0} 
\sqrt{|j_2|+1}(\alpha_2+1)|z_2|^{2\alpha_2}\\
\nonumber &\leq \|T_{\phi}\| \sqrt{|j_2|+1} 
\frac{\sqrt{\alpha_1+j_1+1}}{\sqrt{\alpha_1+1}}. 
\end{align}

Now, we consider the complexified form of $A(\alpha_1, |z_2|)$  as 
\[A(\alpha_1, \xi) =\pi^2 (1-\xi^2)^2 
\sum_{\alpha_2\in \mathbb{N}_0} \frac{\xi^{2\alpha_2}}{(\alpha_1+1) 
	\|z^{\alpha}\|\|z^{\alpha+J}\|}\lambda'_{\alpha}.\]
Let us define
\[f_{\alpha_1}(\xi)=\pi^{-2}(1-z^2)^{-2} A(\alpha_1,\xi) 
=\sum_{\alpha_2\in \mathbb{N}_0} \frac{\xi^{2\alpha_2}}{(\alpha_1+1) 
	\|z^{\alpha}\|\|z^{\alpha+J}\|}\lambda'_{\alpha} \]
for all $\xi\in \mathbb{D}$ and $\alpha_1\in \mathbb{N}_0$. Then, 
$f_{\alpha_1}$ is holomorphic on $\mathbb{D}$
for every $\alpha_1\in \mathbb{N}_0$ 
as $|\lambda'_{\alpha}|\leq \|T_{\phi}\|$ and $(\alpha_1+1) 
	\|z^{\alpha}\|\|z^{\alpha+J}\|\approx \alpha_2^{-2}$. Furthermore,  
$A(\alpha_1,|\xi|)\to 0$ for $\alpha_1\to\infty$ implies that  
$f_{\alpha_1}\to 0$ on the real line in $\mathbb{D}$. 

Next we will show that $f_{\alpha_1}\to 0$ uniformly on 
$\overline{\mathbb{D}}_{1/2}=\{\xi\in \C:|\xi|\leq 1/2\}$ as 
$\alpha_1\to\infty$. First we note that computations in 
 \eqref{EqnBdd} imply $\{A(\alpha_1,\cdot)\}$ is uniformly 
 bounded on $\mathbb{D}$. 
Hence $\{f_{\alpha_1}\}$ is a uniformly bounded sequence 
on $\{\xi\in\C:|\xi|\leq 3/4\}$. Let $\{f_j\}$ 
be a subsequence of $\{f_{\alpha_1}\}$. Montel's theorem 
implies that there exists $\{f_{j_k}\}$, a subsequence of
$f_j$, that is convergent uniformly  on compact subsets 
of $\mathbb{D}_{3/4}$ to a holomorphic function $F$. 
Since $f_{\alpha_1}\to 0$ on the real line in the unit disc 
as $\alpha_1\to\infty$, $F=0$ on the real line in $\mathbb{D}_{3/4}$. 
Then the identity principle for holomorphic functions implies that 
$F=0$ on $\mathbb{D}$. Hence we showed that 
every subsequence of $\{f_{\alpha_1}\}$ has a further subsequence 
converging to zero uniformly on $\overline{\mathbb{D}}_{1/2}$. 
Since uniform convergence on a compact set is metrizable, 
we use the following well known fact to conclude that 
$f_{\alpha_1}\to 0$ uniformly on $\overline{\mathbb{D}}_{1/2}$.  
A sequence $\{x_j\}$ in a metrizable space converges to 
$x$ if and only if every subsequence $\{x_{j_k}\}$ 
has a further subsequence $\{x_{j_{k_s}}\}$ convergent to $x$.  
Finally, we use Cauchy estimates to conclude  that 
\[f^{(2\alpha_2)}_{\alpha_1}(0)
=\frac{(2\alpha_2)!\lambda'_{\alpha}}{(\alpha_1+1)\|z^{\alpha}\|\|z^{\alpha+J}\|}
\to 0 \text{ as }  \alpha_1\to\infty\] for every 
$\alpha_2\in \mathbb{N}_0$. Therefore, $\lambda'_{\alpha}\to 0$  as  
$\alpha_1\to\infty$ for all  $\alpha_2.$  		
\end{proof} 		

We note that the proof above shows that, by switching the roles 
of $\alpha_1$ and $\alpha_2$,  in case $b\D$ contains a non-trivial 
horizontal analytic disc, then $\lambda'_{\alpha}\to 0$ as  
$\alpha_2\to\infty$ for all  $\alpha_1.$  	 

Let $\D$ be a bounded domain in $\C^n$ and $\mu$ be a measure 
supported on $\Dc$.  We say $|k_z|^2dV\to \mu$ weakly as $z_j\to p$ if 
$\int \phi|k_{z_j}|^2dV\to \int \phi d\mu$ as $z_j\to p$ for all 
$\phi\in C(\Dc)$. We will use the following fact in the next lemma. 
If $K$ is a compact metric space, then $C(K)$ is a separable 
Banach space.  Hence the dual of the closed unit ball of 
$C(K)$ is weak-star metrizable (see, for instance, 
\cite[Theorem 5.1 in Ch V]{CBook}). This fact allows us 
to use sequences  with Alaoglu theorem 
in the proof of the lemma below. 

A compact set $K\subset b\D$ is said to be a holomorphic peak 
set if there exists a holomorphic function $F$ on $\D$ that 
is continuous on $\Dc$ such that $F=1$ on $K$ and $|F|<1$ 
on $\Dc\setminus K$.  

\begin{lemma}\label{LemConv}
Let $\D$ be a bounded domain in $\C^n$ , $K\subseteq b\D$ 
be a holomorphic peak set, $p\in K$, and $\phi\in C(\Dc)$ such that 
$\phi=0$ on $K$. Then $\|\phi k_z\|\to 0$ for $z\to p$. 
\end{lemma}

\begin{proof} 
Let $F\in C(\Dc)$ be a holomorphic function on $\D$ such that $F=1$ 
on $K$ and $|F|<1$ on $\Dc\setminus K$. Let us pick is a sequence 
$\{z_j\}\subset \D$ such that $z_j\to p$. Then by 
Alaoglu theorem there exists a subsequence $\{z_{j_k}\}$ such 
that $\{|k_{z_{j_k}}|^2dV\}$ converges to a probability measure 
$\mu$ weakly as $k\to\infty$. Then  
\[1=F(p)\leftarrow F(z_{j_k})=\widetilde{F}(z_{j_k})
=\int F|k_{z_{j_k}}|^2dV \to \int Fd\mu 
\text{ as } k\to\infty.\]
Then one can show that the support of $\mu$ is contained in $K$. 
%
%
%
%
Furthermore, 
\[\left\|\phi k_{z_{j_k}}\right\|^2
=\int|\phi|^2|k_{z_{j_k}}|^2dV\to \int |\phi|^2d\mu
=0  \text{ as } k\to\infty\]
because $\phi=0$ on $K$. 

So we showed that for any sequence $\{z_j\}$ converging to $p$ 
there exists a subsequence  $\{z_{j_k}\}$ such that 
$\|\phi k_{z_{j_k}}\|\to 0$ as $k\to \infty$. One can show 
that this is equivalent to $\|\phi k_z\|\to 0$ as $z\to p$. 
\end{proof}

The following theorem plays a vital role in the proof of 
Theorem \ref{ThmMain1}.

\begin{thm} \label{Prop2} 
Let $\D$ be a bounded convex Reinhardt domain  in $\C^2$ 
and  $\phi\in C(\Dc)$ be a quasi-homogeneous function with 
multi-degree $J\in \mathbb{Z}^2$. Then $\widetilde{T}_{\phi}(z)\to 0$ 
as $z \to b\D$ if and only if $\phi=0$ on $b\D$.
\end{thm}

\begin{proof} 
First we note that bounded convex Reinhardt domains are bounded 
pseudoconvex domains with Lipschitz boundaries as convex functions 
are locally Lipschitz (see, for instance, \cite{WayneState}). 
Then  \cite[Lemma 1]{RS24} implies that $k_z\to 0$ weakly as $z\to b\D$. 
Then $\phi=0$ on $b\D$ implies that $T_{\phi}$ is compact. 
So $\widetilde{T}_{\phi}(z)\to 0$ as $z \to b\D$. 

To prove the converse we will use the following fact:  
a bounded convex Reinhardt domain in $\C^2$ can have vertical 
or horizontal discs only (see  \cite[Lemma 1]{ClS18}). That is, 
if a boundary point is contained in an analytic disc then it is 
contained either in a horizontal or a vertical disc in the boundary. 
Then, without loss of generality, we assume that $\D$ has a vertical 
disc in the boundary. Hence, we assume that there exist bidiscs 
\[\mathbb{D}^2=\{z\in \C^2: |z_1|<1, |z_2|< 1\}  
\text{ and } U=\{z\in \C^2: |z_1|< 1, |z_2|< s\}\]
such that $\mathbb{D}^2\subset \D\subset U$ (see Figure \ref{Fig}). 

Let us assume that $T_{\phi}e_{\alpha}=\lambda'_{\alpha}e_{\alpha+J}$ 
 for $\alpha\in \mathbb{N}_0^2$. By Lemma \ref{LemToepQhomo} 
 we have $\lambda'_{\alpha}\to 0$ as $\alpha_1\to\infty$ for all  $\alpha_2.$ 
		
Next we will show that $\phi=0$ on the vertical discs in the boundary. 
Let $z_0\in \C$ such that $|z_0|=1$ we write 
$\phi=\phi(z_0,\cdot)+\phi-\phi(z_0,\cdot)$. 
Since the vertical disc at $z_1=z_0$ is a holomorphic peak set and   
$\phi-\phi(z_0,\cdot)$ vanish on this disc, Lemma \ref{LemConv} 
implies that 
\[\|(\phi-\phi(z_0,\cdot))k_z\|\to 0 \text{ as } z_1\to z_0.\]
Then 
\[\left|\widetilde{T}_{\phi}(z)-\widetilde{T}_{\phi(z_0,\cdot)}(z)\right| 
\leq \|(\phi-\phi(z_0,\cdot))k_z\|
\to 0 \text{ as } z_1\to z_0.\] 
Hence, $\widetilde{T}_{\phi(z_0,\cdot)}(z)\to 0$ as $z_1\to z_0$. 
Therefore, without loss of generality, we assume that $\phi$ 
is independent of $z_1$.

 We note that  
\[\phi(z_1,z_2)=\varphi(|z_1|,|z_2|)e^{iJ\cdot \theta}\] 
where  $\varphi$ is multi-radial continuous function independent of $|z_1|$. 
That is, $\varphi(|z_1|,|z_2|)=\varphi(1,|z_2|)$. We will show that 
$\varphi(1, \cdot)=0$. 

Lemma \ref{LemToepQhomo} and \eqref{Q} imply that for 
any fixed $\alpha_2$ we have  
\[\lambda'_{\alpha}=\int\varphi(z)|e_{\alpha}(z)||e_{\alpha+J}(z)|dV(z)\to 0 
\text{ as } \alpha_1\to\infty.\] 
We note that $\varphi(1,y)=\varphi(x,y)$  (as $\varphi$ is independent 
of $|z_1|$)  and 
\[\psi'(1,\alpha_2)= \int_0^1\varphi(1,y)y^{2\alpha_2+j_2+1}dy\]
as $\rho_1(1)=1$ where $\psi'$ is defined in \eqref{EqnPsi'}.  

In the following computations we use the inequality $x<\rho_2(y)$ 
to define $|\D|$ in $\mathbb{R}^2$. 
\begin{align*}
& \frac{4\pi^2}{\|z^{\alpha}\|\|z^{\alpha+J}\|}
	\int_0^1x^{2\alpha_1+j_1+1}\psi'(1,\alpha_2)dx \\
&= \frac{\int_0^1\int_0^1
	\varphi(1,y)x^{2\alpha_1+j_1+1}y^{2\alpha_2+j_2+1}dxdy}{
	\left(\int_0^{s}\int_0^{\rho_2(y)}x^{2\alpha_1+1}y^{2\alpha_2+1}dxdy
	\right)^{1/2}\left(\int_0^{s}\int_0^{\rho_2(y)}x^{2\alpha_1+2j_1+1}
	y^{2\alpha_2+2j_2+1}dxdy\right)^{1/2}}\\
&= \frac{\frac{2\sqrt{(\alpha_1+1)(\alpha_1+j_1+1)}}{2\alpha_1+j_1+2}
	\int_0^1\varphi(1,y)y^{2\alpha_2+j_2+1}
	dy}{\left(\int_0^{s}y^{2\alpha_2+1}(\rho_2(y))^{2\alpha_1+2}dy\right)^{1/2}
	\left(\int_0^{s}y^{2\alpha_2+2j_2+1}(\rho_2(y))^{2\alpha_1+2j_1+2}dy
	\right)^{1/2}}.
\end{align*}
Next we use the fact that $0\leq \rho_2(y)< 1$ for $1<y\leq s$ 
and $\rho_2(y)= 1$ for $0\leq y\leq 1$ together 
with the Lebesgue's dominated convergence theorem to conclude that 
\begin{align*}
\int_0^{s}y^{2\alpha_2+1}(\rho_2(y))^{2\alpha_1+2}dy
& \to \int_0^1y^{2\alpha_2+1}dy =\frac{1}{2\alpha_2+2}\\
	\int_0^{s}y^{2\alpha_2+2j_2+1}(\rho_2(y))^{2\alpha_1+2j_1+2}dy
& \to \int_0^1y^{2\alpha_2+2j_2+1}dy
	=\frac{1}{2\alpha_2+2j_2+2} 
\end{align*}
as  $\alpha_1\to \infty.$ Then 
\begin{align} \nonumber 
\frac{4\pi^2}{\|z^{\alpha}\|\|z^{\alpha+J}\|}&
	\int_0^1x^{2\alpha_1+j_1+1}\psi'(1,\alpha_2)dx  \\
\label{Eqn5}&\to 2\sqrt{(\alpha_2+1)(\alpha_2+j_2+1)}
	\int_0^1\varphi(1,y)y^{2\alpha_2+j_2+1}dy
\end{align} 
as $\alpha_1\to\infty$. 

Next we write
\begin{align*}
\frac{4\pi^2}{\|z^{\alpha}\|\|z^{\alpha+J}\|}\int_0^1x^{2\alpha_1+j_1+1}
	\psi'(x,\alpha_2)dx 
=&\, \frac{4\pi^2}{\|z^{\alpha}\|\|z^{\alpha+J}\|}\int_0^1x^{2\alpha_1+j_1+1}
	\psi'(1,\alpha_2)dx \\
&\,+\frac{4\pi^2}{\|z^{\alpha}\|\|z^{\alpha+J}\|}\int_0^1x^{2\alpha_1+j_1+1}
	(\psi'(x,\alpha_2)-\psi'(1,\alpha_2))dx.  
\end{align*} 
One can show that, using \eqref{EqnRatio},  the second term on the 
right hand side above goes to zero as $\alpha_1\to\infty$. 
Then \eqref{Eqn5} and Lemma \ref{LemToepQhomo}  imply that 
\begin{align*}
0\leftarrow \lambda'_{\alpha}
=&\, \frac{4\pi^2}{\|z^{\alpha}\|\|z^{\alpha+J}\|}
	\int_0^1x^{2\alpha_1+j_1+1}\psi'(x,\alpha_2)dx\\
\to &\, 2\sqrt{(\alpha_2+1)(\alpha_2+j_2+1)}
	\int_0^1\varphi(1,y)y^{2\alpha_2+j_2+1}	dy
\end{align*} 
 as  $\alpha_1\to\infty$
for any fixed $\alpha_2$.  Then 
\[0=\psi'(1,\alpha_2)= \int_0^1\varphi(1,y)y^{2\alpha_2+j_2+1}dy\] 
for every $\alpha_2$. However, an argument using Stone-Weierstrass theorem 
implies that $\varphi(1,y)=0$ for all $0\leq y\leq 1$. 

If there is a horizontal disc, the same proof works by switching the roles of 
$\alpha_1$ and $\alpha_2$.  To finish the proof, we need to take care of
the points that are not on a horizontal or vertical disc. By 
\cite[Proposition 3.2]{FS98}) such points are holomorphic peak points. 
Then we use \cite[Lemma 15]{CSZ18}  to conclude 
that $\phi=\widetilde{T}_{\phi}=0$ on such points. Therefore, 
$\phi=0$ on $b\D$. 
\end{proof}

\begin{proof}[Proof of Theorem \ref{ThmMain1}]
Let us assume that 
$\widetilde{T_{\phi}}(z)\to 0$ for $z\to b\D$. Then 
Corollary \ref{CorZero} implies that 
$\widetilde{T_{\phi_J}}(z)\to 0$ for $z\to b\D$ for all 
$J\in \mathbb{Z}^2$. Then Theorem \ref{Prop2} implies 
that $\phi_J=0$ on $b\D$ for all $J\in \mathbb{Z}^2$. 
Then the Ces\`aro's mean $\Lambda_k(\phi)$ vanishes on 
$b\D$ for all $k$ and, by Lemma \ref{LemSum}, 
$\{\Lambda_k(\phi)\}$ converges to $\phi$ uniformly on $\Dc$ 
as $k\to\infty$. Therefore, $\phi=0$ on $b\D$.

The converse is a result of the facts that $T_{\phi}$ is compact 
whenever $\phi=0$ on $b\D$ and  $k_z\to 0$ weakly for $z\to b\D$ 
(see, \cite[Lemma 1]{RS24}). 
\end{proof}

\begin{proof}[Proof of Corollary \ref{CorToeplitz}]
(i) $\Leftrightarrow$ (ii) is a true statement, in general. Also 
(i) $\Leftrightarrow$ (v) has been proven in \cite{RS24}. 
Next, (ii) $\Rightarrow$ (iii)  because $k_z\to 0$ weakly as 
$z\to b\D$. The statement (iii) $\Rightarrow$ (iv) is true because 
\[|\widetilde{T}_{\phi}(z)|^2\leq \|T_{\phi}k_z\|^2
=\widetilde{T^*_{\phi}T_{\phi}}(z).\]  
Finally, (iv) $\Rightarrow$ (v) follows from Theorem \ref{ThmMain1}. 
\end{proof}

The following proposition is implicit in \cite{RS24}. We state 
it here explicitly as it can be of interest on its own. 
\begin{prop}\label{Prop1}
Let $\D$ be a bounded pseudoconvex  in $\C^n$ with Lipschitz 
boundary  and $\phi\in C(\overline{\D})$. Then the following are 
equivalent:
\begin{itemize}
	\item[(i)] $T_{\phi}$ is compact.
	\item[(ii)] $T_{|\phi|}$ is compact.
	\item[(iii)] $\widetilde{T_{|\phi|}}(z)\to 0$ as $z\to b\D$.
	\item[(iv)] $\phi=0$ on $b\D$.
\end{itemize}
\end{prop}

\begin{proof} 
First we note that \cite[Theorem 1]{RS24} implies that (i) and (iv) 
as well as (ii) and (iv) are equivalent. Also (ii) $\Rightarrow$ (iii) is a result of 
the fact that $k_z\to 0$ weakly as $z\to b\D$ (see \cite[Lemma 1]{RS24}). 
Then to complete the proof, it is enough to show that (iii) $\Rightarrow$ (iv).  
This is included implicitly in the last part of the proof of \cite[Theorem 1]{RS24}, 
in which, the following is proven: Let $\D\subset \C^n$ be a bounded
pseudoconvex domain  with Lipschitz boundary and $\phi\in C(\Dc)$  
such that $\phi\geq 0$. Then  $\widetilde{\phi}=0$ on $b\D $ 
implies that $\phi=0$ on $b\D$.
\end{proof}

\section{Proof of Theorem \ref{ThmMain2} and 
Corollary \ref{CorHankel}}

\begin{lemma}\label{LemHankelSquareQhomo}
Let $\D$ be a bounded convex Reinhardt domain  in $\C^2$ 
and  $\phi\in C(\Dc)$ be a quasi-homogeneous function with 
multi-degree $J\in \mathbb{Z}^2$. Assume that $b\D$ contains 
a non-trivial vertical analytic disc and 
$\widetilde{H_{\phi}^*H_{\phi}}(z)\to 0$ as $z\to b\D$. 
Then $\lambda_{\alpha}\to 0$ as  $\alpha_1\to\infty$ 
for all  $\alpha_2\in\mathbb{N}_0$ where 
$H_{\phi}^*H_{\phi}e_{\alpha}=\lambda_{\alpha}e_{\alpha}$ for  
$\alpha\in \mathbb{N}_0^2$.    	
\end{lemma} 

\begin{proof} 
This first part of the proof is very similar to the proof of 
Lemma \ref{LemToepQhomo}, without loss of generality, we assume that 
\[\mathbb{D}^2=\{z\in \C^2: |z_1|<1, |z_2|< 1\}  
\text{ and } U=\{z\in \C^2: |z_1|< 1, |z_2|< s\}\]
such that $\mathbb{D}^2\subset \D\subset U$ (see Figure \ref{Fig}). 
We note that $\lambda_{\alpha}\geq 0$ and \eqref{EqnEigH} implies that 
\[	\lambda_{\alpha}=\begin{cases}
\lambda''_{\alpha} \text{ for } 
	\alpha+J\not\in \mathbb{N}_0^2\\ 
 \lambda''_{\alpha}-|\lambda'_{\alpha}|^2 \text{  for } 
	\alpha+J\in \mathbb{N}_0^2
\end{cases}  \]
for $\alpha\in \mathbb{N}_0^2$. 
One can calculate the Berezin transform of $H_{\phi}^*H_{\phi}$ 
as follows  
\begin{align*}
\widetilde{H_{\phi}^*H_{\phi}}(z)
&
	= \frac{1}{K(z,z)} \langle H_{\phi}^*H_{\phi} K(\cdot,z), K(\cdot,z) \rangle \\
& =\frac{1}{K(z,z)} \sum_{\alpha,\beta \in \mathbb{N}^2_0} 
	\overline{e_{\alpha}(z)} e_{\beta}(z) \langle H_{\phi}^*H_{\phi}e_{\alpha}, 
	e_{\beta}\rangle\\
& =\frac{K_{\mathbb{D}^2}(z,z)}{K(z,z)} \frac{1}{K_{\mathbb{D}^2}(z,z)} 
	\sum_{\alpha\in \mathbb{N}^2_0}
	\frac{|z^{\alpha}|^2}{\|z^{\alpha}\|^2}\lambda_{\alpha}. 
\end{align*}
Here $\frac{K_{\mathbb{D}^2}(z,z)}{K(z,z)}\geq 1$ since 
$\mathbb{D}^2\subset \D$ (see \cite[p. 413]{JPBook}). 
Then $\widetilde{H_{\phi}^*H_{\phi}}(z)\to 0$ as $z\to b\D$ implies that 
\[ \frac{1}{K_{\mathbb{D}^2}(z,z)}\sum_{\alpha\in \mathbb{N}^2_0}
\frac{|z^{\alpha}|^2}{\|z^{\alpha}\|^2}\lambda_{\alpha} \to 0  
\text{ as } z\to b\D.\]
We write 
\begin{align*}
\frac{1}{K_{\mathbb{D}^2}(z,z)}\sum_{\alpha\in \mathbb{N}^2_0}
\frac{|z^{\alpha}|^2}{\|z^{\alpha}\|^2}\lambda_{\alpha} 
=\left(1-|z_1|^2\right)^2 
	\sum_{\alpha_1\in \mathbb{N}_0}(\alpha_1+1)
	 A(\alpha_1, |z_2|)|z_1|^{2\alpha_1}
\end{align*}
where
\[ A(\alpha_1,|z_2|)
=\pi^2 (1-|z_2|^2)^2
\sum_{\alpha_2\in \mathbb{N}_0} \frac{
|z_2|^{2\alpha_2}\lambda_{\alpha}}{(\alpha_1+1)\|z^{\alpha}\|^2}.\]
Let us denote $t=|z_1|^2$. Then 
\begin{align*}
\lim_{t\to 1^{-}}(1-t)^2
\sum_{\alpha_1\in  \mathbb{N}_0} (\alpha_1+1)A(\alpha_1,|z_2|)t^{\alpha_1}=0.
\end{align*}
Computations similar to \eqref{EqnBdd} implies that there 
exists $C>0$ such that  $0\leq A(\alpha_1, |z_2|)<C$ for all 
$\alpha_1\in \mathbb{N}_0$ and $|z_2|<1$.  
We note that 
\begin{align*}
(1-t)^2
&\sum_{\alpha_1\in  \mathbb{N}_0} (\alpha_1+1)A(\alpha_1,|z_2|)t^{\alpha_1} 
	=(1-t)A(0,|z_2|)\\
&+ (1-t)\left(\sum_{\alpha_1\in  \mathbb{N}}  \left((\alpha_1+1)A(\alpha_1,|z_2|)
	- \alpha_1A(\alpha_1-1,|z_2|)\right)t^{\alpha_1}\right).  
\end{align*} 
Next we want to show that 
\begin{align}\label{EqnBdd1Prod} 
\sup\{|(\alpha_1+1)A(\alpha_1,|z_2|)- \alpha_1A(\alpha_1-1,|z_2|)|:
\alpha_1\in  \mathbb{N}\}<\infty
\end{align}  
for any $z_2\in \mathbb{D}$. We note that 
\begin{align*}
&|(\alpha_1+1)A(\alpha_1,|z_2|)- \alpha_1A(\alpha_1-1,|z_2|)| \\
&\leq
	\pi^2 (1-|z_2|^2)^2 
	\sum_{\alpha_2\in \mathbb{N}_0} |z_2|^{2\alpha_2} 
	\left|\frac{\lambda_{\alpha_1,\alpha_2}}{\|z^{\alpha}\|^2}
	-\frac{\lambda_{\alpha_1-1,\alpha_2}}{\|z^{(\alpha_1-1,\alpha_2)}\|^2}\right|. 
\end{align*} 
Since $\lambda_{\alpha}=\lambda''_{\alpha}-|\lambda'_{\alpha}|^2$ 
we get 
\[\left|\frac{\lambda_{\alpha_1,\alpha_2}}{\|z^{\alpha}\|^2}
-\frac{\lambda_{\alpha_1-1,\alpha_2}}{\|z^{(\alpha_1-1,\alpha_2)}\|^2}\right|
\leq 
\left|\frac{\lambda''_{\alpha_1,\alpha_2}}{\|z^{\alpha}\|^2}
-\frac{\lambda''_{\alpha_1-1,\alpha_2}}{\|z^{(\alpha_1-1,\alpha_2)}\|^2}\right|
+\left|\frac{|\lambda'_{\alpha_1,\alpha_2}|^2}{\|z^{\alpha}\|^2}
-\frac{|\lambda'_{\alpha_1-1,\alpha_2}|^2}{\|z^{(\alpha_1-1,\alpha_2)}\|^2}\right|.\]
Then to prove \eqref{EqnBdd1Prod} it is enough to show that 
\begin{align}
\label{EqnSup1}&\sup\left\{\left|\frac{\lambda''_{\alpha_1,\alpha_2}}{
	\|z^{\alpha}\|^2}-\frac{\lambda''_{\alpha_1-1,\alpha_2}}{
\|z^{(\alpha_1-1,\alpha_2)}\|^2}\right|:\alpha_1\in  \mathbb{N}\right\}<\infty,\\
\label{EqnSup2}&\sup\left\{\left|\frac{|\lambda'_{\alpha_1,\alpha_2}|^2}{
	\|z^{\alpha}\|^2}-\frac{|\lambda'_{\alpha_1-1,\alpha_2}|^2}{
	\|z^{(\alpha_1-1,\alpha_2)}\|^2}\right|:\alpha_1\in  \mathbb{N}\right\}<\infty.
\end{align}   

Below we use the inequality$y<\rho_1(x)$ to 
define  $|\D|=\{(|z_1|,|z_2|)\in\mathbb{R}^2:z\in \D\}$. 
\begin{align*}
\lambda''_{\alpha} 
=&\, \frac{1}{\|z^{\alpha}\|^2}\int|\varphi(z)|^2|z^{\alpha}|^2dV(z) \\
=&\,\frac{4\pi^2}{\|z^{\alpha}\|^2} \int_0^1x^{2\alpha_1+1} 
	\int_0^{\rho_1(x)}|\varphi(x,y)|^2y^{2\alpha_2+1}dydx\\
=&\,\frac{4\pi^2}{\|z^{\alpha}\|^2}
	\int_0^1x^{2\alpha_1+1} \psi''(x,\alpha_2)dx
\end{align*}
where  
 $\psi''(x,\alpha_2)=\int_0^{\rho_1(x)}|\varphi(x,y)|^2y^{2\alpha_2+1}dy$. 
Hence,  
\begin{align*}
&\frac{\lambda''_{\alpha_1,\alpha_2}}{\|z^{\alpha}\|^2}
	-\frac{\lambda''_{\alpha_1-1,\alpha_2}}{\|z^{(\alpha_1-1,\alpha_2)}\|^2} \\
&=\frac{4\pi^2}{\|z^{\alpha}\|^4}
	\int_0^1x^{2\alpha_1+1}\psi''(x,\alpha_2)dx
	-\frac{4\pi^2}{\|z^{(\alpha_1-1,\alpha_2)}\|^4}
	\int_0^1x^{2\alpha_1-1}\psi''(x,\alpha_2)dx\\
&=\frac{4\pi^2}{\|z^{\alpha}\|^4}
\int_0^1(x^{2\alpha_1+1}-x^{2\alpha_1-1})\psi''(x,\alpha_2)dx\\
&+\left(\frac{4\pi^2}{\|z^{\alpha}\|^4}
	-\frac{4\pi^2}{\|z^{(\alpha_1-1,\alpha_2)}\|^4}\right) 
\int_0^1x^{2\alpha_1-1}\psi''(x,\alpha_2)dx.
\end{align*} 
The first term on the right hand side of the inequality is bounded 
because  integral is dominated by $\alpha_1^{-2}$ and 
$\|z^{\alpha}\|^4\approx \alpha_1^{-2}$. 
The second integral on the right hand side is dominated by $\alpha_1^{-1}$.
On the other hand, one can use \eqref{EqnRatio} to show that 
\begin{align*}
&\frac{1}{\|z^{\alpha}\|^4}-\frac{1}{\|z^{(\alpha_1-1,\alpha_2)}\|^4} \\
=& \frac{1}{\|z^{\alpha}\|^4}- \frac{(\alpha_1+1)^2\pi^4}{(\alpha_2+1)^2}
	-\frac{1}{\|z^{(\alpha_1-1,\alpha_2)}\|^4}
	+\frac{\alpha_1^2\pi^4}{(\alpha_2+1)^2}
	+\frac{(2\alpha_1+1)\pi^4}{(\alpha_2+1)^2}
\approx\,  \alpha_1. 
\end{align*} 	
Therefore we verified \eqref{EqnSup1}. 
Next we will verify \eqref{EqnSup2}. 
\begin{align} 
\nonumber &\frac{|\lambda'_{\alpha_1,\alpha_2}|^2}{
	\|z^{\alpha}\|^2}-\frac{|\lambda'_{\alpha_1-1,\alpha_2}|^2}{
	\|z^{(\alpha_1-1,\alpha_2)}\|^2}\\
\nonumber & =	\frac{|\lambda'_{\alpha_1,\alpha_2}|^2}{\|z^{\alpha}\|^2}
	-\frac{|\lambda'_{\alpha_1,\alpha_2}||\lambda'_{\alpha_1-1,\alpha_2}|}{
	\|z^{\alpha}\|^2}+
	\frac{|\lambda'_{\alpha_1,\alpha_2}||\lambda'_{\alpha_1-1,\alpha_2}|}{
		\|z^{\alpha}\|^2}
	-\frac{|\lambda'_{\alpha_1-1,\alpha_2}|^2}{
		\|z^{(\alpha_1-1,\alpha_2)}\|^2}\\
\label{EqnH1}& =|\lambda'_{\alpha_1,\alpha_2}| 	\left(\frac{|\lambda'_{\alpha_1,\alpha_2}|
	-|\lambda'_{\alpha_1-1,\alpha_2}|}{\|z^{\alpha}\|^2}\right)
+\left(\frac{|\lambda'_{\alpha_1,\alpha_2}|}{
	\|z^{\alpha}\|^2}
-\frac{|\lambda'_{\alpha_1-1,\alpha_2}|}{
	\|z^{(\alpha_1-1,\alpha_2)}\|^2}\right)|\lambda'_{\alpha_1-1,\alpha_2}|. 	
\end{align} 
We note that $|\lambda'_{\alpha}|\leq \|\phi\|_{L^{\infty}(\D)}<\infty$ and 
$\psi'$ is bounded.  Let us focus on the first term on the right hand 
side in \eqref{EqnH1}.  
\begin{align*}
&\left|\frac{|\lambda'_{\alpha_1,\alpha_2}|
	-|\lambda'_{\alpha_1-1,\alpha_2}|}{\|z^{\alpha}\|^2} \right| \\
&	\leq\left|\frac{\lambda'_{\alpha_1,\alpha_2}
		-\lambda'_{\alpha_1-1,\alpha_2}}{\|z^{\alpha}\|^2} \right| \\
&=\left|\frac{4\pi^2}{\|z^{\alpha}\|^3\|z^{\alpha+J}\|}
\int_0^1x^{2\alpha_1+j_1+1}\psi'(x,\alpha_2)dx \right.\\
&\left. - \frac{4\pi^2}{\|z^{\alpha}\|^2\|z^{(\alpha_1-1,\alpha_2)}\|
	\|z^{(\alpha_1-1,\alpha_2)+J}\|}
\int_0^1x^{2\alpha_1+j_1-1}\psi'(x,\alpha_2)dx \right|  \\
&\leq \frac{4\pi^2}{\|z^{\alpha}\|^3\|z^{\alpha+J}\|}
\int_0^1|x^{2\alpha_1+j_1+1}-x^{2\alpha_1+j_1-1}| |\psi'(x,\alpha_2)|dx\\
&+\left|  \frac{4\pi^2}{\|z^{\alpha}\|^3\|z^{\alpha+J}\|}
-  \frac{4\pi^2}{\|z^{\alpha}\|^2\|z^{(\alpha_1-1,\alpha_2)}\|
	\|z^{(\alpha_1-1,\alpha_2)+J}\|}\right|
	\int_0^1x^{2\alpha_1+j_1-1}|\psi'(x,\alpha_2)|dx \\
	&\lesssim \alpha_1^2\int_0^1(x^{2\alpha_1+j_1-1}-x^{2\alpha_1+j_1+1})dx\\
&	+ \frac{1}{\alpha_1\|z^{\alpha}\|^2}\left|  \frac{1}{\|z^{\alpha}\|\|z^{\alpha+J}\|}
	-  \frac{1}{\|z^{(\alpha_1-1,\alpha_2)}\| 	\|z^{(\alpha_1-1,\alpha_2)+J}\|}\right|.
\end{align*} 
The first term on the right hand side above is bounded 
as $\int_0^1(x^{2\alpha_1+j_1-1}-x^{2\alpha_1+j_1+1})dx
\approx \alpha_1^{-2}$. The second term is bounded because  
\begin{align*}
	& \|z^{(\alpha_1-1,\alpha_2)}\|\|z^{(\alpha_1-1,\alpha_2)+J}\|
	-  \|z^{\alpha}\|\|z^{\alpha+J}\|\\
	& =\frac{\|z^{(\alpha_1-1,\alpha_2)}\|^2\|z^{(\alpha_1-1,\alpha_2)+J}\|^2
	-\|z^{\alpha}\|^2\|z^{\alpha+J}\|^2}{\|z^{(\alpha_1-1,\alpha_2)}\|
	\|z^{(\alpha_1-1,\alpha_2)+J}\|+ \|z^{\alpha}\|\|z^{\alpha+J}\|}\\
\text{(by \eqref{EqnNormDiff})}	&\approx \frac{\alpha_1^{-3}}{\alpha^{-1}_1}=\alpha_1^{-2}.
\end{align*}  
Hence the first term on the right hand side in \eqref{EqnH1} is bounded. 
To prove the second term we use similar inequalities below. 
\begin{align*} 
&\left|\frac{|\lambda'_{\alpha_1,\alpha_2}|}{\|z^{\alpha}\|^2}
-\frac{|\lambda'_{\alpha_1-1,\alpha_2}|}{\|z^{(\alpha_1-1,\alpha_2)}\|^2}\right|\\
	&\lesssim \frac{1}{\|z^{\alpha}\|^3\|z^{\alpha+J}\|}
	\int_0^1|x^{2\alpha_1+j_1+1}-x^{2\alpha_1+j_1-1}| |\psi'(x,\alpha_2)|dx\\
	&+\left|  \frac{1}{\|z^{\alpha}\|^3\|z^{\alpha+J}\|}
	-  \frac{1}{\|z^{(\alpha_1-1,\alpha_2)}\|^3
		\|z^{(\alpha_1-1,\alpha_2)+J}\|}\right|
	\int_0^1x^{2\alpha_1+j_1-1}|\psi'(x,\alpha_2)|dx \\
	&\lesssim \alpha_1^2\int_0^1(x^{2\alpha_1+j_1-1}-x^{2\alpha_1+j_1+1})dx\\
	&	+ \frac{1}{\alpha_1}\left|  \frac{1}{\|z^{\alpha}\|^3\|z^{\alpha+J}\|}
	-  \frac{1}{\|z^{(\alpha_1-1,\alpha_2)}\|^3 	\|z^{(\alpha_1-1,\alpha_2)+J}\|}\right|.
\end{align*} 
Again the first term on the right hand side above is bounded as 
the integral is dominated by $\alpha_1^{-2}$. 
The second term is bounded if and only if 
\begin{align*}
 \|z^{(\alpha_1-1,\alpha_2)}\|^3\|z^{(\alpha_1-1,\alpha_2)+J}\|
-  \|z^{\alpha}\|^3\|z^{\alpha+J}\|
\approx \alpha^{-3}_1.
\end{align*} 
However, 
\begin{align*}
&\|z^{(\alpha_1-1,\alpha_2)}\|^3\|z^{(\alpha_1-1,\alpha_2)+J}\|
-  \|z^{\alpha}\|^3\|z^{\alpha+J}\|\\
&=\frac{\|z^{(\alpha_1-1,\alpha_2)}\|^6\|z^{(\alpha_1-1,\alpha_2)+J}\|^2
- \|z^{\alpha}\|^6\|z^{\alpha+J}\|^2}{\|z^{(\alpha_1-1,\alpha_2)}\|^3
\|z^{(\alpha_1-1,\alpha_2)+J}\|+ \|z^{\alpha}\|^3\|z^{\alpha+J}\|}.
\end{align*} 
Since $\|z^{(\alpha_1-1,\alpha_2)}\|^3
\|z^{(\alpha_1-1,\alpha_2)+J}\|+ \|z^{\alpha}\|^3\|z^{\alpha+J}\|\approx \alpha_1^{-2}$, 
it is enough to show that 
\begin{align*}
\|z^{(\alpha_1-1,\alpha_2)}\|^6\|z^{(\alpha_1-1,\alpha_2)+J}\|^2
-  \|z^{\alpha}\|^6\|z^{\alpha+J}\|^2
\approx \alpha^{-5}_1.
\end{align*} 
To verify the estimate above we write
\begin{align*} 
& \|z^{(\alpha_1-1,\alpha_2)}\|^6\|z^{(\alpha_1-1,\alpha_2)+J}\|^2
-  \|z^{\alpha}\|^6\|z^{\alpha+J}\|^2\\
	\approx &\left(\int_0^1x^{2\alpha_1-1}
(\rho_1(x))^{2\alpha_2+2}dx\right)^3 \left(\int_0^1x^{2\alpha_1+2j_1-1}
(\rho_1(x))^{2\alpha_2+2j_2+2}dx\right) \\
& -\left(\int_0^1x^{2\alpha_1+1}
(\rho_1(x))^{2\alpha_2+2}dx\right)^3\left(\int_0^1x^{2\alpha_1+2j_1+1}
(\rho_1(x))^{2\alpha_2+2j_2+2}dx\right) .
\end{align*}
Let us denote 
\begin{align*} 
X_1=&\int_0^1(x^{2\alpha_1-1}-x^{2\alpha_1+1})
(\rho_1(x))^{2\alpha_2+2}dx \approx \alpha_1^{-2},\\
X_2=&\int_0^1x^{2\alpha_1+1} (\rho_1(x))^{2\alpha_2+2}dx
\approx \alpha_1^{-1},\\
X_3=&\int_0^1x^{2\alpha_1+2j_1-1} (\rho_1(x))^{2\alpha_2+2j_2+2}dx
\approx \alpha_1^{-1},\\
X_4=&\int_0^1(x^{2\alpha_1+2j_1-1}-x^{2\alpha_1+2j_{1}+1})
(\rho_1(x))^{2\alpha_2+2j_2+2}dx \approx \alpha_1^{-2}.
\end{align*}  
Then 
\begin{align*}
&\|z^{(\alpha_1-1,\alpha_2)}\|^6\|z^{(\alpha_1-1,\alpha_2)+J}\|^2
	-  \|z^{\alpha}\|^6\|z^{\alpha+J}\|^2\\
&=(X_1+X_2)^3X_3-X^3_2(X_3-X_4)\\
&=X_1^3X_3+3X_1^2X_2X_3+3X_1X_2^2X_3+X_2^3X_4\\
&\approx \alpha_1^{-7}+\alpha_1^{-6}+\alpha_1^{-5}+\alpha_1^{-5}
\approx \alpha_1^{-5}.
\end{align*}
Hence, we verified  that  \eqref{EqnBdd1Prod} 
we finish the proof the same as in the proof of 
Lemma \ref{LemToepQhomo}.  
\end{proof}

\begin{prop} \label{Prop3} 
Let $\D$ be a bounded convex Reinhardt domain  in $\C^2$ 
and  $\phi\in C(\Dc)$ be a quasi-homogeneous function with 
multi-degree $J\in \mathbb{Z}^2$. Then 
$\widetilde{H^*_{\phi}H_{\phi}}(z)=\|H_{\phi}k_z\|^2\to 0$
as $z\to b\D$ if and only if $\phi\circ f$ is holomorphic for 
any holomorphic $f:\mathbb{D}\to b\D$.
\end{prop}

\begin{proof} 
We will prove the forward direction only as the converse is a consequence 
of  \cite{ClS18}. So we assume that $\widetilde{H^*_{\phi}H_{\phi}}(z)=\|H_{\phi}k_z\|^2\to 0$
as $z\to b\D$. 
Since $\phi$ is a quasi-homogeneous function with multi-degree 
$J\in \mathbb{Z}^2$, we have 
$H^*_{\phi}H_{\phi}e_{\alpha}=\lambda_{\alpha}e_{\alpha}$ 
for $\alpha\in\mathbb{N}_0^2$. 

Without loss of generality, let us assume that $b\D$ has a vertical discs 
at $z_0=1$ where $|z_0|=1$ and below we will use the inequality 
$y<\rho_1(x)$ to define $|\D|=\{(|z_1|,|z_2|)\in\mathbb{R}^2:z\in \D\}$. 

Lemma \ref{LemHankelSquareQhomo} implies that 
$\lambda_{\alpha}\to 0$ as  $\alpha_1\to\infty$ for all $\alpha_2$. 
First we deal with the case $\alpha_2+j_2\in \mathbb{N}_0$ and 
$\alpha_1$ large enough, we have 
\begin{align*}
\lambda_{\alpha}=&\int_{\D} |\varphi(z)|^2 |e_{\alpha}(z)|^{2}dV(z) 
	- \left|\int_{\D} \varphi(z) |e_{\alpha}(z)|| e_{\alpha+J}(z)|dV(z)\right|^2\\
=&\,\frac{1}{\|z^{\alpha}\|^{2}} \int_{\D} |\varphi(z)|^{2} |z^{\alpha}|^{2}dV(z)
	-  \frac{1}{\|z^{\alpha}\|^2 \|z^{\alpha+J}\|^{2}} 
	\left|\int_{\D} \varphi(z) |z^{\alpha}|| z^{\alpha+J}|dV(z)\right|^2\\
=& \, \frac{\int_0^1\int_0^{\rho_1(x)}|\varphi(x,y)|^2x^{2\alpha_1+1}
	y^{2\alpha_2+1}dydx}{\int_0^1\int_0^{\rho_1(x)}x^{2\alpha_1+1}
	y^{2\alpha_2+1}dydx}\\
&-\frac{\left|\int_0^1\int_0^{\rho_1(x)} \varphi(x,y)x^{2\alpha_1+j_1+1}
	y^{2\alpha_2+j_2+1}dydx\right|^2}{\int_0^1\int_0^{\rho_1(x)}
	x^{2\alpha_1+1}y^{2\alpha_2+1}dydx\int_0^1\int_0^{\rho_1(x)}
	x^{2\alpha_1+2j_1+1}y^{2\alpha_2+2j_2+1}dydx}\\
=&\, \frac{(2\alpha_2+2)\int_0^1F(x)x^{2\alpha_1+1}dx}{
	\int_0^1(\rho_1(x))^{2\alpha_2+2}x^{2\alpha_1+1}dx}\\
&-\frac{(2\alpha_2+2)(2\alpha_2+2j_2+2)\left|\int_0^1G(x)x^{2\alpha_1+j_1+1}dx\right|^2}{
	\int_0^1(\rho_1(x))^{2\alpha_2+2}x^{2\alpha_1+1}dx
	\int_0^1(\rho_1(x))^{2\alpha_2+2j_2+2}x^{2\alpha_1+2j_2+1}dx} 
\end{align*}
where  
\begin{align*}
F(x)&= \int_0^{\rho_1(x)}|\varphi(x,y)|^2y^{2\alpha_2+1}dy,\\
G(x)&= \int_0^{\rho_1(x)}\varphi(x,y)y^{2\alpha_2+j_2+1}dy.
\end{align*} 
Next using \cite[Lemma 3.3]{Le10H} we have the following limits
\begin{align*}
\frac{\int_0^1F(x)x^{2\alpha_1+1}dx}{
	\int_0^1(\rho_1(x))^{2\alpha_2+2}x^{2\alpha_1+1}dx}
&=\frac{\int_0^1F(x)x^{2\alpha_1+1}dx}{
	\int_0^1x^{2\alpha_1+1}dx}
	\frac{\int_0^1x^{2\alpha_1+1}dx}{
	\int_0^1(\rho_1(x))^{2\alpha_2+2}x^{2\alpha_1+1}dx}\\
&\to \frac{F(1)}{(\rho_1(1))^{2\alpha_2+2}} =F(1) \text{ as } \alpha_1\to\infty
\end{align*} 
as $\rho_1(1)=1$. Similarly,
\begin{align*}
&\frac{\left|\int_0^1G(x)x^{2\alpha_1+j_1+1}dx\right|^2}{
	\int_0^1(\rho_1(x))^{2\alpha_2+2}x^{2\alpha_1+1}dx
	\int_0^1(\rho_1(x))^{2\alpha_2+2j_2+2}x^{2\alpha_1+2j_2+1}dx} \\
&\to \frac{|G(1)|^2}{(\rho_1(1))^{2\alpha_2+2}(\rho_1(1))^{2\alpha_2+2j_2+2}} 
=|G(1)|^2
\end{align*} 
as $\alpha_1\to \infty.$ Then $\lambda_{\alpha}\to 0$ as $\alpha_1\to \infty$ 
implies that 
\[ |G(1)|^2 =\frac{F(1)}{2\alpha_2+2j_2+2}.\] 
That is, 
\begin{align*} 
	\left|\int_0^{1}\varphi(1,y)y^{2\alpha_2+j_2+1}dy\right|^2
=&\, \frac{1}{2\alpha_2+2j_2+2}
	\int_0^{1}|\varphi(1,y)|^2y^{2\alpha_2+1}dy \\
=&\, \int_0^{1}y^{2\alpha_2+2j_2+1}dy \int_0^{1}|\varphi(1,y)|^2y^{2\alpha_2+1}dy.
\end{align*}
Then for $f=\varphi(1,y)y^{\alpha_2+1/2}$ and $g= y^{\alpha_2+j_2+1/2}$ 
we have $|\langle f,g\rangle|=\|f\|\|g\|$. Since we have equality 
in the Cauchy-Schwarz inequality, $f$ and $g$ must be multiples of each other. 
That is, $\varphi(1,y)y^{\alpha_2+1/2}=ay^{\alpha_2+j_2+1/2}$ 
for some constant $a$. Hence $\varphi(1,y)=ay^{j_2}$. 
Namely, $\phi(z_0,z_2)=az_2^{j_2}$ is holomorphic.  
Therefore, $\phi$ is holomorphic along the vertical disc at $z_0$.

Next we consider the case $\alpha_2+j_2\notin \mathbb{N}_0$. 
Similar to the previous case we have   
\begin{align*}
\lambda_{\alpha}=& \int_{\D} |\varphi(z)|^{2} |e_{\alpha}(z)|^{2}dV(z) \\
=&\, \frac{(2\alpha_2+2)\int_0^1F(x)x^{2\alpha_1+1}dx}{
	\int_0^1(\rho_1(x))^{2\alpha_2+2}x^{2\alpha_1+1}dx}
	\to F(1) \text{ as } \alpha_1\to\infty.
\end{align*}
Then $\lambda_{\alpha}\to 0$ as $\alpha_1\to \infty$ 
implies that 
\[ F(1)=\int_0^1|\varphi(1,y)|^2y^{2\alpha_2+1}dy=0.\] 
Hence, $\varphi(1,y)=0$. That is, in either case $\phi$ is holomorphic 
along the vertical discs. Similarly, if $b\D$ contains horizontal discs 
a similar argument shows that $\phi$ is holomorphic along the 
horizontal discs. Therefore, 
$\phi\circ f$ is holomorphic for any holomorphic $f:\mathbb{D}\to b\D$.
\end{proof} 

We note that when $\phi\circ f$ is holomorphic for any holomorphic 
$f:\mathbb{D}\to b\D$ we say that $\phi$ is  holomorphic along discs in $b\D$. 

\begin{proof}[Proof of Theorem \ref{ThmMain2}]
Let us assume that $\widetilde{H^*_{\phi}H_{\phi}}(z)\to 0$ for $z\to b\D$. 
Then Lemma \ref{Lem2Quasihomo} implies that 
$\widetilde{H^*_{\phi_J}H_{\phi_J}}(z)\to 0$ for $z\to b\D$ for all $J$. 
Then Proposition \ref{Prop3} implies that $\phi_J$ is holomorphic 
along discs in $b\D$ for all $J$. Then the Ces\`aro's mean 
$\Lambda_k(\phi)$ holomorphic along discs in $b\D$ for all $k$ and, 
by Lemma \ref{LemSum}, $\{\Lambda_k(\phi)\}$ converges to $\phi$ 
uniformly on $\Dc$ as $k\to\infty$. Therefore, $\phi$ is holomorphic 
along discs in $b\D$. The converse of the theorem is 
due \cite[Theorem 1]{ClS18}. 
\end{proof}

\begin{proof}[Proof of Corollary \ref{CorHankel}]
(i) $\Leftrightarrow$ (ii) is true in general and  (i) $\Leftrightarrow$ (iv) is proven in 
\cite{ClS18}. Furthermore, (ii) $\Rightarrow$ (iii) is true because $k_z\to 0$ 
weakly as $z\to b\D$. Finally, (iii) $\Rightarrow$ (iv) by Theorem \ref{ThmMain2}.
\end{proof}

\section*{Acknowledgments}
This work was carried out while the first author visited University of Toledo, 
supported by The Scientific and Technological Research Council of Turkey 
(TUBITAK), under the grant 1059B192201455.


\end{document}